\documentclass{amsart}
\usepackage[hiresbb]{graphicx}
\usepackage{amsmath,amssymb}

\newtheorem{df}{Definition}[section]
\newtheorem{thm}[df]{Theorem}
\newtheorem{prop}[df]{Proposition}
\newtheorem{lemm}[df]{Lemma}
\newtheorem{cor}[df]{Corollary}

\newtheorem{rem}[df]{Remark}

\newcommand{\id}{\mathrm{id}}
\newcommand{\Q}{\mathbb{Q}}
\newcommand{\R}{\mathbb{R}}
\newcommand{\Z}{\mathbb{Z}}
\newcommand{\C}{\mathbb{C}}

\newcommand{\shuugou}[1]{\{ #1 \}}
\newcommand{\zettaiti}[1]{\lvert #1 \rvert}

\newcommand{\gyaku}[1]{ #1^{-1}}

\newcommand{\skein}[1]{\mathcal{S}( #1 )}

\newcommand{\kukakko}[1]{\langle #1 \rangle}
\newcommand{\defeq}{\stackrel{\mathrm{def.}}{=}}

\newcommand{\filtn}[1]{\{ #1 \}_{n \geq 0}}
\newcommand{\ii}{\sqrt{-1}}

\newcommand{\arccosh }{\mathrm{arccosh}}

\begin{document}

\title{Dehn twists on Kauffman bracket skein algebras}
\author{Shunsuke Tsuji}
\date{}
\address{Graduate School of Mathematical Sciences, The University of Tokyo, 
3-8-1 Komaba, Meguro-ku, Tokyo 153-8941, Japan}
\subjclass[2010]{57N05, 57M27}
\email{tsujish@ms.u-tokyo.ac.jp}
\keywords{Mapping class group, Kauffman bracket, Dehn twist, Skein algebra}
\maketitle

\begin{abstract}
We give an explicit formula for the action of the Dehn twist
along a simple closed curve  in a compact connected oriented surface
 on the completion of the filtered skein modules.
To do this, we introduce filtrations
of the Kauffman bracket skein algebra and the Kauffman 
bracket skein modules on the surface.
\end{abstract}

\section{Introduction}

Recently it has come to light that the Goldman Lie algebra of a surface plays an 
important role in the study of the mapping class group of the surface.
See \cite{Kawazumi}, \cite{KK} and \cite{MT} for details.
Before that, Turaev \cite{Turaev} drew an analogy between the 
Goldman Lie algebra and some skein algebra.
Hence it is important to establish some explicit
connection between the Kauffman bracket skein algebra and the mapping class group.
This new connection motivates much of the interest in 
the mapping class group of a surface and the link theory.
In fact, skein algebras give us a new way of
studying the mapping class group.
Furthermore, we expect that this connection
will bring us some information about $3$-manifolds
including the Casson invariant.

The aim of this paper is to explain a new relationship between the Kauffman bracket 
skein algebra and the mapping class group. Let $\Sigma$ be a 
compact connected oriented surface with non-empty boundary. Kawazumi-Kuno \cite{KK}
\cite{Kawazumi} defined an action $\sigma$ of the Goldman Lie algebra on
the group ring of the fundamental group of $\Sigma$.
Using this action, Kawazumi-Kuno \cite{KK} \cite{Kawazumi} and
Massuyeau-Turaev \cite{MT} proved the formula for the action of the 
Dehn twist $t_c$ along a simple closed curve $c$
\begin{equation}
\label{KawazumiKuno}
t_c =\exp (\sigma(\frac{1}{2} \zettaiti{(\log(c))^2})) :\widehat{\Q \pi_1(\Sigma)} \to
\widehat{\Q \pi_1 (\Sigma)}.
\end{equation}
Our goal in this paper is to establish a skein algebra version of this formula.

Let $\Sigma$ be a compact connected oriented surface, 
$I$ the closed interval $[0,1]$ and $\Q [A, \gyaku{A}] $
the ring of Laurent polynomials over $\Q$ in an indeterminate $A$.
The Kauffman bracket skein algebra $\skein{\Sigma}$ is defined to be
the quotient of the free $\Q [A,\gyaku{A}]$-module
with basis the set of unoriented framed links in $\Sigma \times I$
by the skein relation which defines the Kauffman bracket.
 Let $J$ be a finite subset of $\partial \Sigma$.
The Kauffman bracket skein module $\skein{\Sigma,J}$ is defined to be 
the quotient of the free $\Q [A, \gyaku{A}]$-module with basis $\mathcal{T}(\Sigma,J)$,
where we denote by $\mathcal{T}(\Sigma,J)$ the set of unoriented framed tangles 
with the base point set $J \times \shuugou{\frac{1}{2}}$.
For details, see the subsection \ref{subsection_define_skein_module}.
The Kauffman skein algebra $\skein{\Sigma}$ has the structure of an associative algebra and 
a Lie algebra over $\Q[A, \gyaku{A}]$.
 The Kauffman bracket skein module $\skein{\Sigma,J}$ has the structure
of a $\skein{\Sigma}$-bimodule. Furthermore, we define the action $\sigma$
of $\skein{\Sigma}$ on $\skein{\Sigma,J}$ such that $\skein{\Sigma,J}$ is 
$\skein{\Sigma}$-module under the action $\sigma$ when we regard 
$\skein{\Sigma}$ as a Lie algebra. For details, see the subsection
 \ref{subsection_Poisson_structure}.
In this paper, we introduce the filtration $\filtn{F^n \skein{\Sigma}}$ of 
$\skein{\Sigma}$ and the filtration $\filtn{F^n \skein{\Sigma,J}}$ of 
$\skein{\Sigma,J}$ defined by an  augmentation ideal.
These operations are continuous
in the topologies of $\skein{\Sigma}$ and $\skein{\Sigma,J}$ induced by
these filtrations. 
We remark that there is some relationship between the completion of
the group ring of  the fundamental group of $\Sigma$ and these filtrations
of $\skein{\Sigma}$ and $\skein{\Sigma,J}$ which will appear in \cite{TsujiCSAII}.
We denote the completions of $\skein{\Sigma} $ and
$\skein{\Sigma,J}$ in these topologies by $\widehat{\skein{\Sigma}}$
and $\widehat{\skein{\Sigma,J}}$ respectively. For details, see  the subsection
\ref{subsection_filtrations_and_completions}.
The main result of the paper is the formula 
for the action of the Dehn twist along a simple closed curve $c$
\begin{equation*}
t_c = \exp (\frac{-A+\gyaku{A}}{4 \log (-A)} (\arccosh   (-\frac{c}{2}))^2) 
:\widehat{\skein{\Sigma,J}} \to \widehat{\skein{\Sigma,J}}.
\end{equation*}
which is a skein version of  the formula (\ref{KawazumiKuno}).
Here $\log(-A) =\sum_{i=1}^\infty \frac{-1}{i}(A+1)^i \in \Q [[A+1]]$
and $(\arccosh (\frac{-c}{2}))^2 =\sum_{i=0}^\infty\frac{i! i!}{(i+1)(2i+1)!}(1-\frac{c^2}{4})^{i+1}
 \in \Q [[c+2]]$.
This skein version does not follow from the original one 
\cite{Kawazumi} \cite{KK} \cite{MT}.

In section 5, we prove the following three propeties.
\begin{enumerate}
\item Let $\Sigma$ be a compact connected oriented surface
with non-empty boundary.
The topology on $\skein{\Sigma,J}$ introduced by the filtration
is Hausdorff, in other words, we have $\cap_{n=0}^\infty F^n \skein{\Sigma,J}=0$.
\item 
Let $\Sigma$ and $\Sigma'$ be two oriented compact connected surfaces
satisfying $\pi_1 (\Sigma) \simeq \pi_1 (\Sigma')$, 
$J$ and $J'$ finite subsets of $\partial \Sigma$ and
$\partial \Sigma'$, respectively,
satisfying $\sharp J = \sharp J'$.
There exists a diffeomorphism $\xi : (\Sigma \times I ,J \times I)
 \to (\Sigma' \times I, J' \times I)$.
Then we have $\xi (F^n \skein{\Sigma,J}) = F^n \skein{\Sigma' ,J'}$.
But the induced map $\xi :\skein{\Sigma} \to \skein{\Sigma'}$
does not seems to be an algebra homomorphism.
\item We have
\begin{equation*}
\sum_{L' \subset L} (-1)^{\zettaiti{L'}}(-2)^{-\zettaiti{L'}}[L'] \in (\ker \epsilon)^n
\end{equation*}
 for a link $L$ in $\Sigma \times I$ 
having components more than $n$,
where the sum is over all sublinks $\emptyset \subseteqq L' \subseteqq L$
 and $\zettaiti{L}$
the number of components of $L$.
In other words, for a link $L$ in $\Sigma \times I$,
$(-2)^{-\zettaiti{L'}}[L'] \mod (\ker \epsilon)^n$
is a finite type invariant of order $n$ in the sense of Le
\cite{Le1996} (3.2).
\end{enumerate}
We need the first property
to prove that the action of the mapping class group of
a  compact connected oriented 
surface with non-empty boundary
on the completed skein algebra of the surface
is faithful.
The second and third properties follows from
Lemma \ref{lemm_new_filtration}.
Using the second property and 
Lickorish's theorem \cite{skeins_and_handlebodies} (Theorem \ref{thm_Lickorish}),
we prove the first property.
In subsequent papers, we need all
the above properties.

In particular, we need the third property
to prove that the action of the mapping class group of
a  compact connected oriented 
surface with non-empty boundary
on the completed skein algebra of the surface.

In subsequent papers, using this formula of Dehn twists,
we obtain an embedding of the Torelli group
into the completed skein algebra defined in this paper.
This embedding gives a  construction of 
the first Johnson homomorphism and 
a new filtration 
consisting of normal subgroups in
the mapping class group.
Furthermore, it gives an invariant for integral homology $3$-spheres
inducing a finite type invariants of any order $n$.
The details will appear  elsewhere \cite{TsujiCSAII}.

\bigskip
\subsection*{Acknowledgment}

The author is grateful to his adviser, Nariya Kawazumi, for helpful discussion
and encouragement and thanks to Kazuo Habiro, Gw\'{e}na\"{e}l Massuyeau, Jun Murakami
and Tomotada Ohtsuki for helpful comments.
In particular, the author would like to thank to Gw\'{e}na\"{e}l Massuyeau
for reading the author's draft and valuable comments to revise it.
This work was supported by JSPS KAKENHI Grant Number 15J05288.

\tableofcontents

\section{Definition of tangles in $\Sigma \times I$}
\label{section_def_tangles}

In this section, let $\Sigma $ be a compact connected oriented surface.

We define the set of tangles in $\Sigma \times I$.

\begin{df}

Let $J$ be a finite subset of $\partial \Sigma $.
We define $\mathcal{E} (\Sigma, J)$ to be the set consisting
of all the 
injective map  $E = \coprod_{i}\tau_i \sqcup \coprod_j \upsilon_j$
from a domain $D$ consisting of a finite collection of strips 
$\coprod_i I \times (-\epsilon,\epsilon)$ and annuli
$\coprod_j S^1 \times (-\epsilon,\epsilon)$ into $\Sigma \times (0,1)$
satisfying the following four conditions.

\begin{enumerate}
\item Each $\upsilon_j$ is an embedding into $\Sigma \times (0,1)$.
\item The restriction of each $\tau_i$ to $(0,1) \times (-\epsilon,\epsilon)$
is an embedding into $\Sigma \times (0,1)$.
\item The restriction of each $\tau_i$ to $\shuugou{0,1} \times (-\epsilon,
\epsilon)$  is an orientation preserving embedding into $J \times I$.
\item For $j \in J $, $\tau (D) \cap (j \times I)$ is not empty and is connected.
\end{enumerate}

Two elements $E_0$ and $E_1$ of $\mathcal{E} (\Sigma, J)$
which have same domain $D$
are unoriented-isotopic if there exists a continuous map
$H : D \times I \to \Sigma \times I $ such that $H(D \times \shuugou{0})
 =E_0(D ) $,
 $H(D \times \shuugou{1})=E_1(D)$ and $H(\cdot ,t) \in 
\mathcal{E}(\Sigma, J)$ for $t \in I$. We denote by $\mathcal{T}(\Sigma, J)$ 
the set of 
unoriented-isotopy
classes of elements of $\mathcal{E} (\Sigma, J)$.
We denote by $\kukakko{\cdot}$ the quotient map $\mathcal{E} (\Sigma, J) \to 
\mathcal{T}(\Sigma,  J)$.
If $J =\emptyset$, we simply denote $\mathcal{T}(\Sigma, J)$ and 
$\mathcal{E}(\Sigma, J)$ by 
$\mathcal{T}(\Sigma)$ and $\mathcal{E}(\Sigma)$.
\end{df}

The definition of `tangles' is similar to the definition of `link'
of marked surfaces in \cite{Mu2012}.
But, a tangle in this definition has one arc on each point of $J$.

\begin{df}

Let $J$ be a finite subset of $\partial \Sigma$.
An element $E$ of $\mathcal{E}(\Sigma, J)$ is
generic if $E :(\coprod_i I
 \sqcup \coprod_j S^1) \times (-\epsilon,\epsilon) \to \Sigma \times I$
 satisfies the following two conditions.
 
 \begin{enumerate}
 \item For $x \in \coprod_i I \sqcup \coprod_j S^1$,
 $(-\epsilon, \epsilon) \to I, t \mapsto p_2 \circ E(x,t)$ is an orientation preserving
 embedding map, where we denote by $p_1$ the projection 
$\Sigma \times I \to I $.
 \item $\coprod_i I \sqcup \coprod_j S^1 \to \Sigma, x \mapsto p_1 \circ
E (x,0)$ is an immersion such that the intersections of the image consist of 
transverse double points, where we denote by $p_2$ the projection
$\Sigma \times I \to \Sigma$.
\end{enumerate}

 \end{df}

It is convenient to present tangles in $\Sigma \times I$
by tangle diagrams on $\Sigma$ in the same fashion
in which links in $\R^3$ may be presented by planar link diagrams.

\begin{df}

Let $J$ be a finite subset of $\partial \Sigma$,
$T$ an element of $\mathcal{T}(\Sigma,J)$ and 
$E :(\coprod_i I
 \sqcup \coprod_j S^1) \times (-\epsilon,\epsilon) \to \Sigma \times I$
 an element of $\mathcal{E}(\Sigma,J)$ such that
$E$ is generic and that $\kukakko{E} =T$.
The tangle diagram of $T$ is $p_1 \circ  E ((\coprod_i I \sqcup \coprod_j
S^1 ) \times \shuugou{0})$
together with height-information, i.e., the choice of the upper branch
of the curve at each crossing.
The chosen branch is called an over crossing; the other branch is 
called an under crossing.

\end{df}

\begin{prop}[see, for example, \cite{BuZi85}]

Let $J$ be a finite subset of $\partial \Sigma$.
Let $T$ and $T'$ be two elements of $\mathcal{T}(\Sigma, J)$
and $d$ and $d'$ tangle diagrams of them respectively. Then, $T$ equals 
$T'$ if and only if $d$ can be transformed into $d'$ by a sequence of isotopies of 
$\Sigma$ and the RI, RII, RIII moves shown in Figure \ref{fig:RI},
\ref{fig:RII}, and \ref{fig:RIII}.

\end{prop}

\begin{figure}
\begin{picture}(300,80)
\put(0,0){\includegraphics[width=60pt]{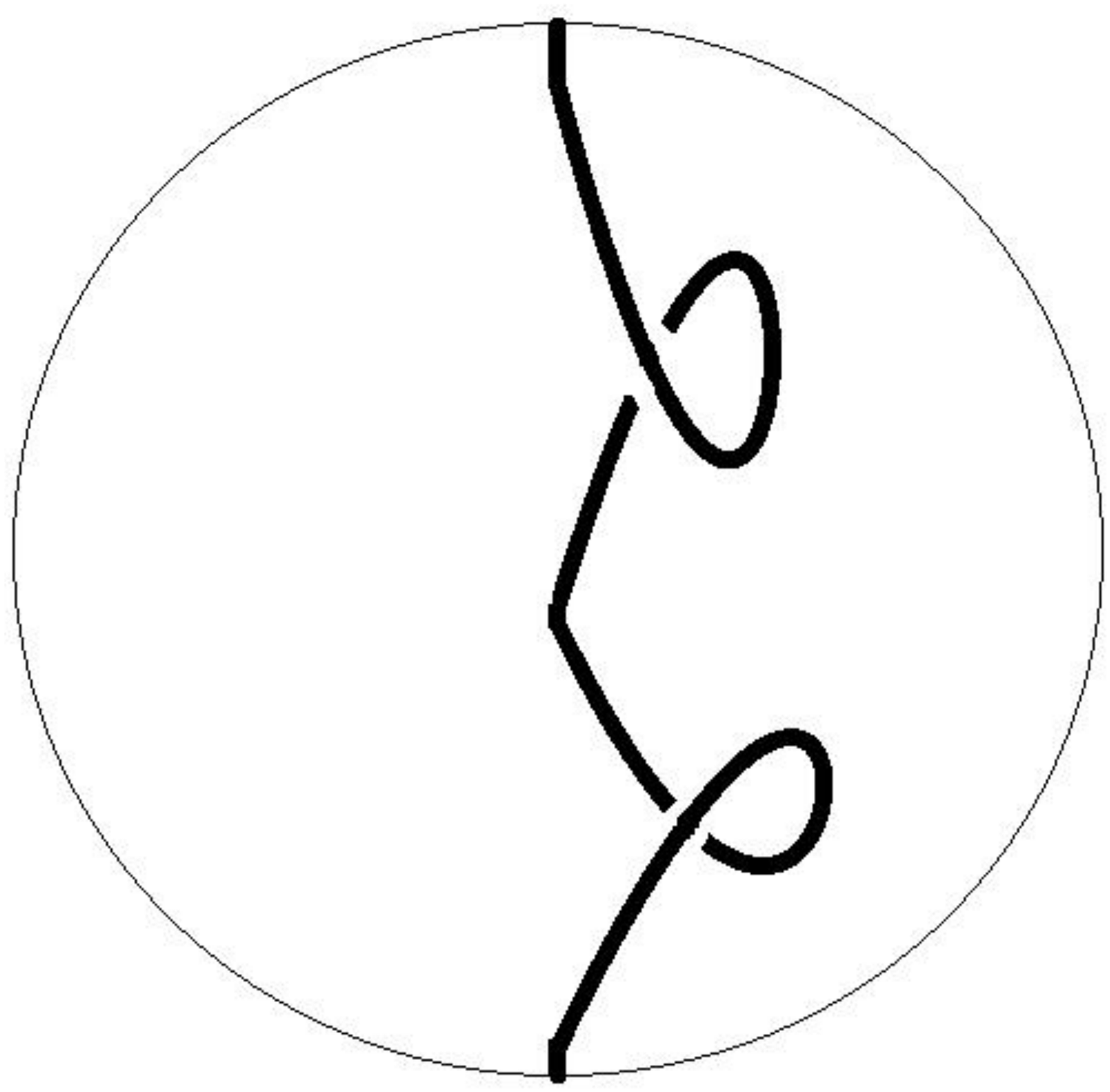}}
\put(70,0){\includegraphics[width=60pt]{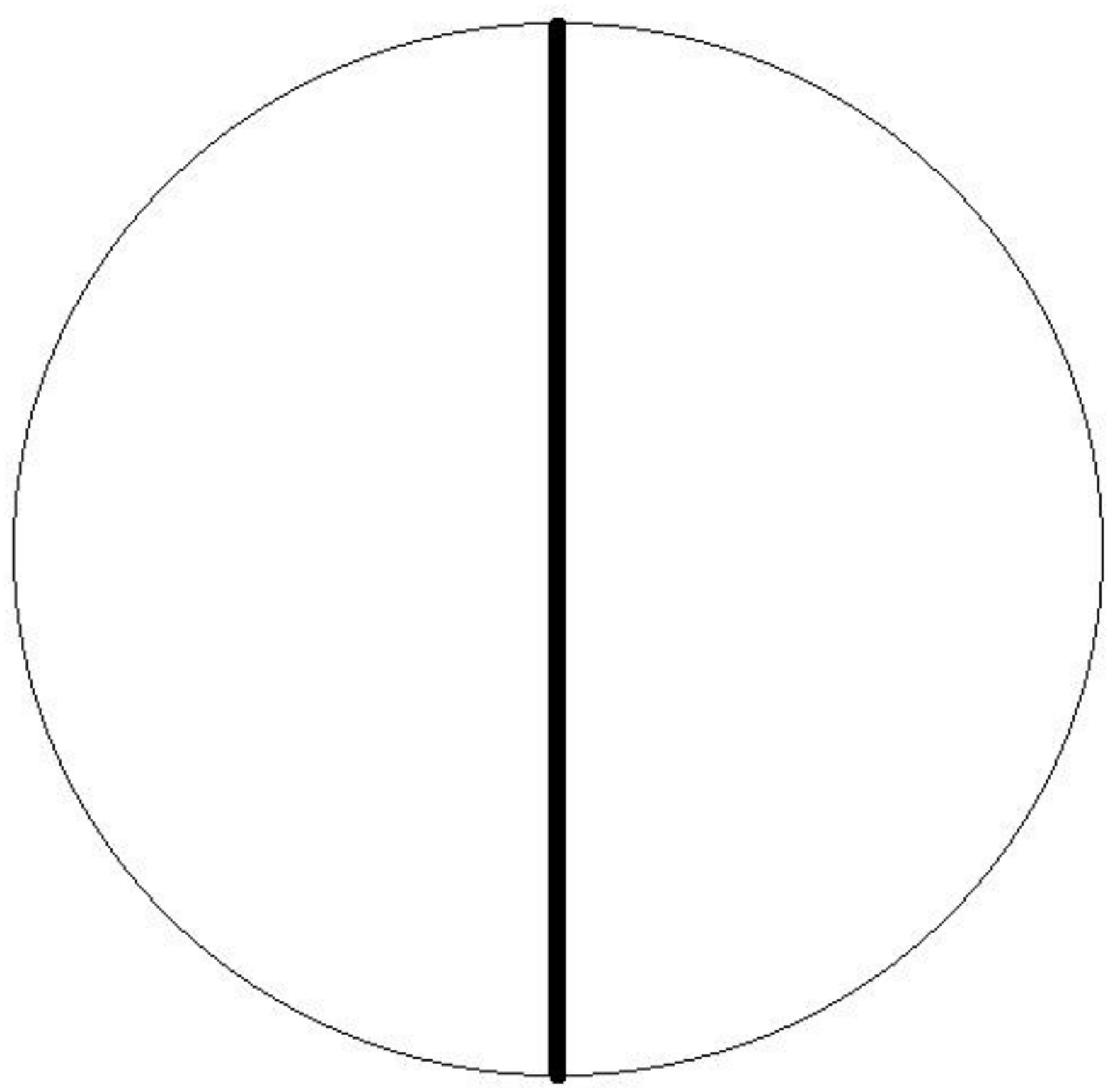}}
\put(140,0){\includegraphics[width=60pt]{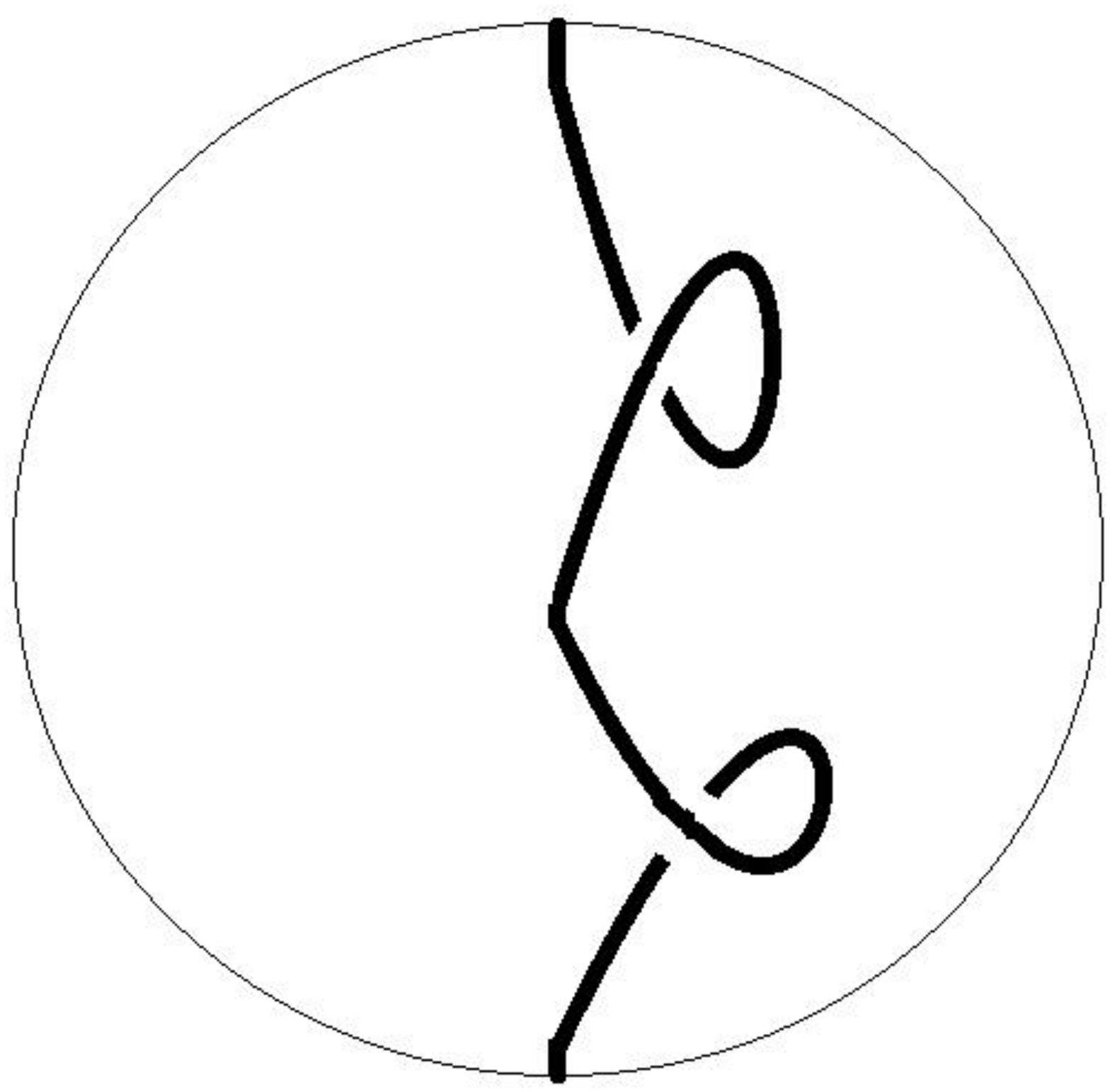}}
\put(60,30){\LARGE{$\leftrightarrow$}}
\put(130,30){\LARGE{$\leftrightarrow$}}
\end{picture}
\caption{RI: Reidemester move I}
\label{fig:RI}
\end{figure}

\begin{figure}
\begin{tabular}{rr}
\begin{minipage}{0.5 \hsize}

\begin{picture}(300,80)
\put(0,0){\includegraphics[width=60pt]{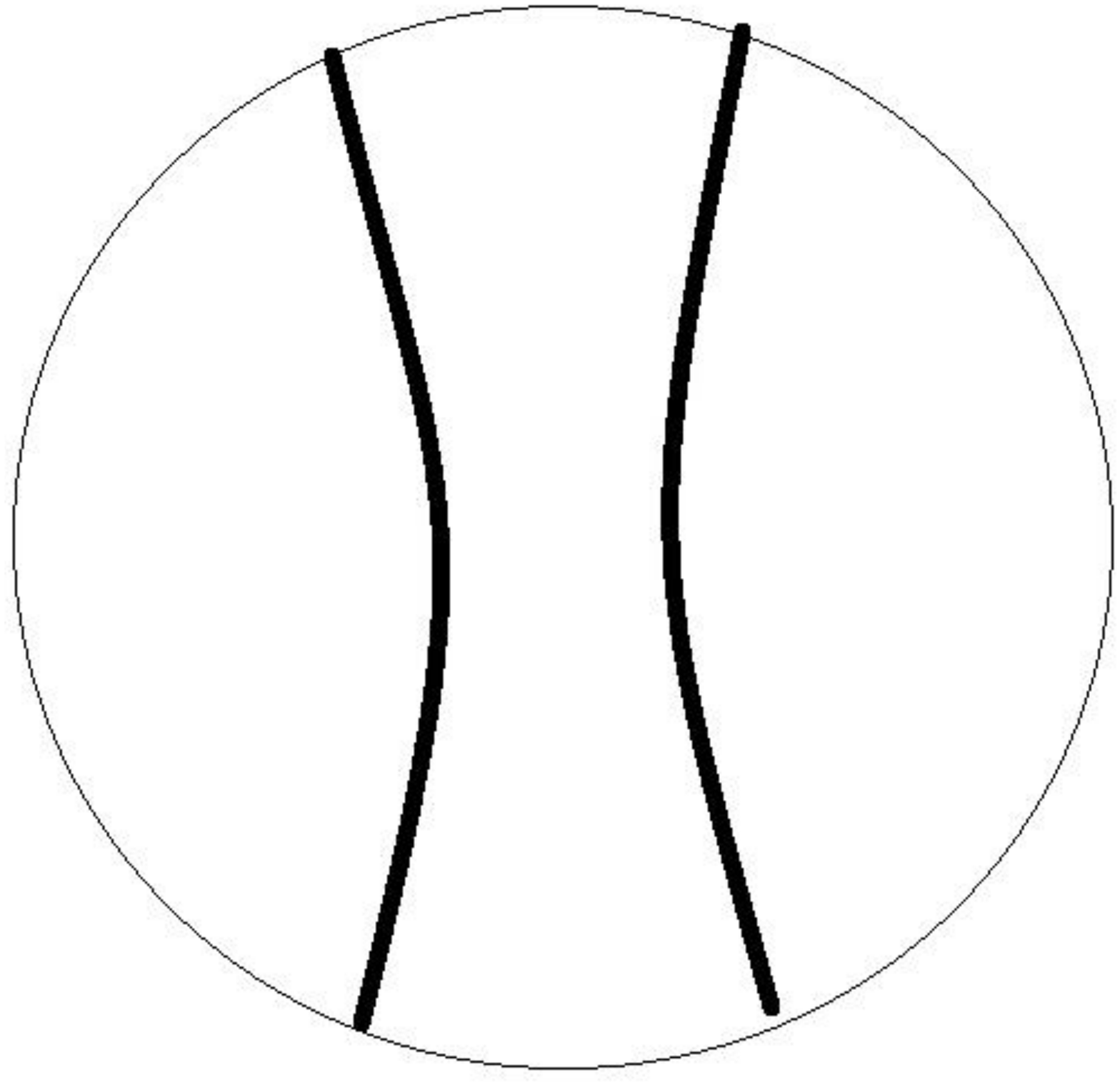}}
\put(70,0){\includegraphics[width=60pt]{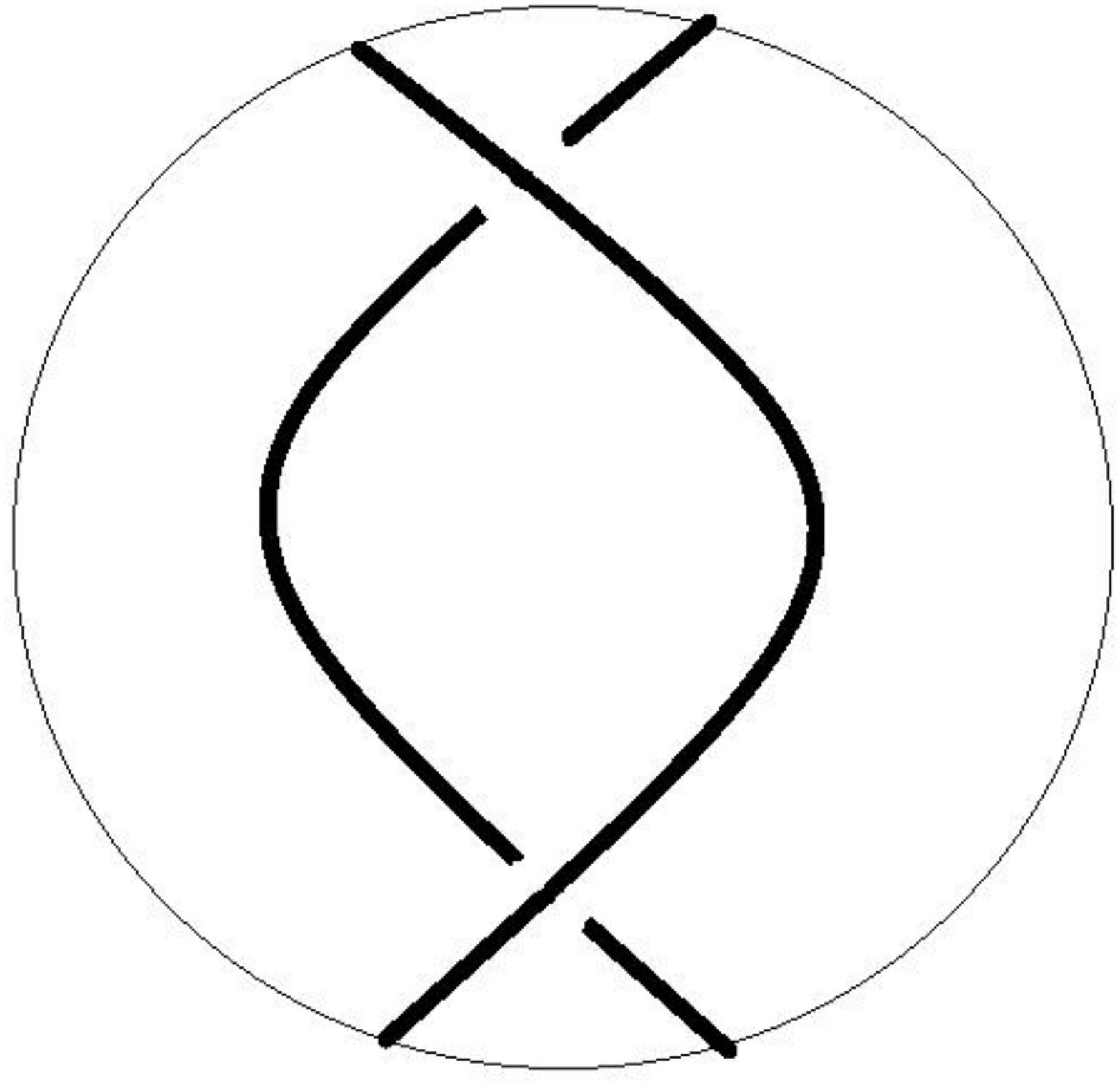}}
\put(60,30){\LARGE{$\leftrightarrow$}}
\end{picture}
\caption{RII: Reidemeister move II}
\label{fig:RII}

\end{minipage}

\begin{minipage}{0.5 \hsize}

\begin{picture}(300,80)
\put(0,0){\includegraphics[width=60pt]{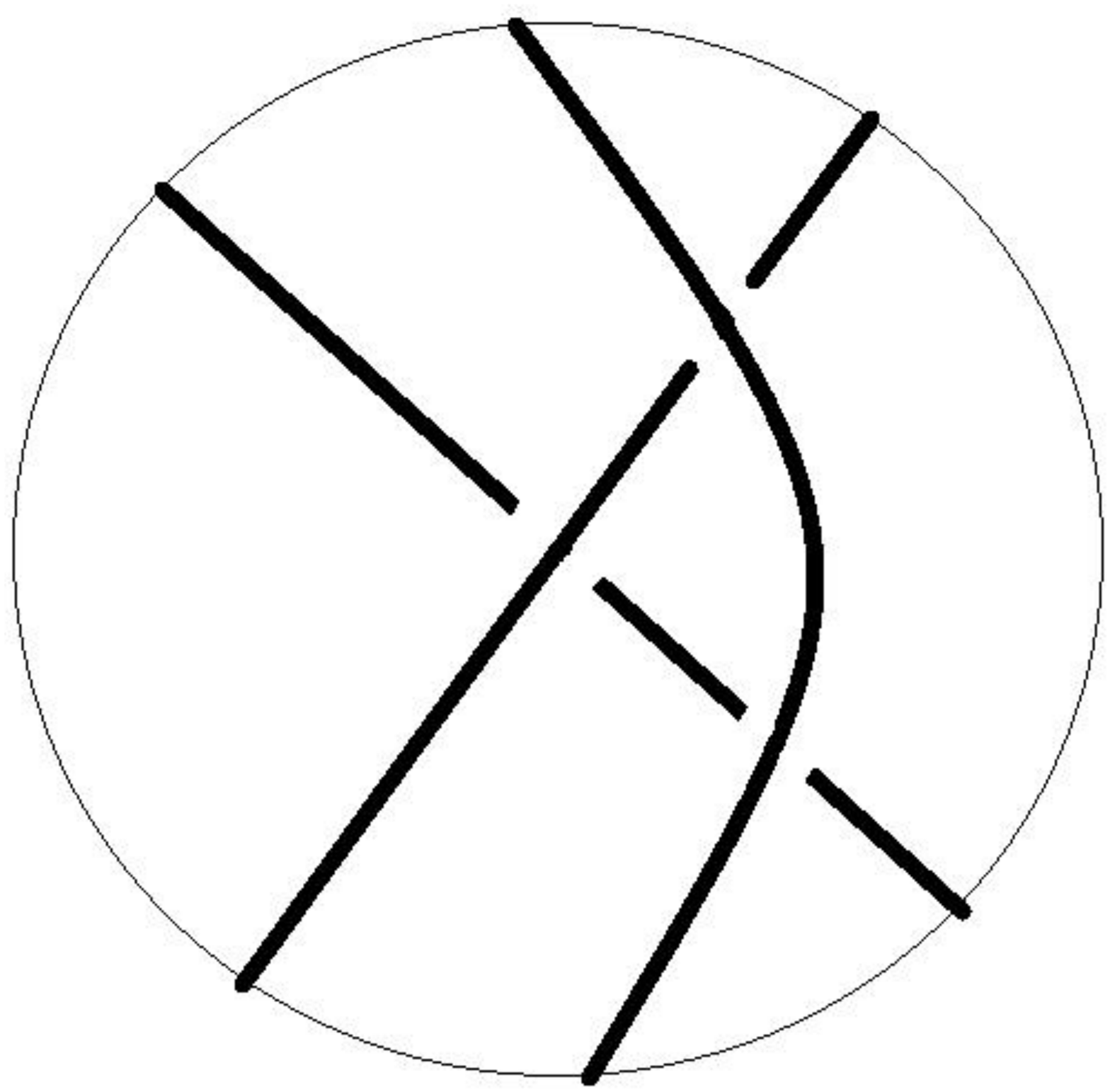}}
\put(70,0){\includegraphics[width=60pt]{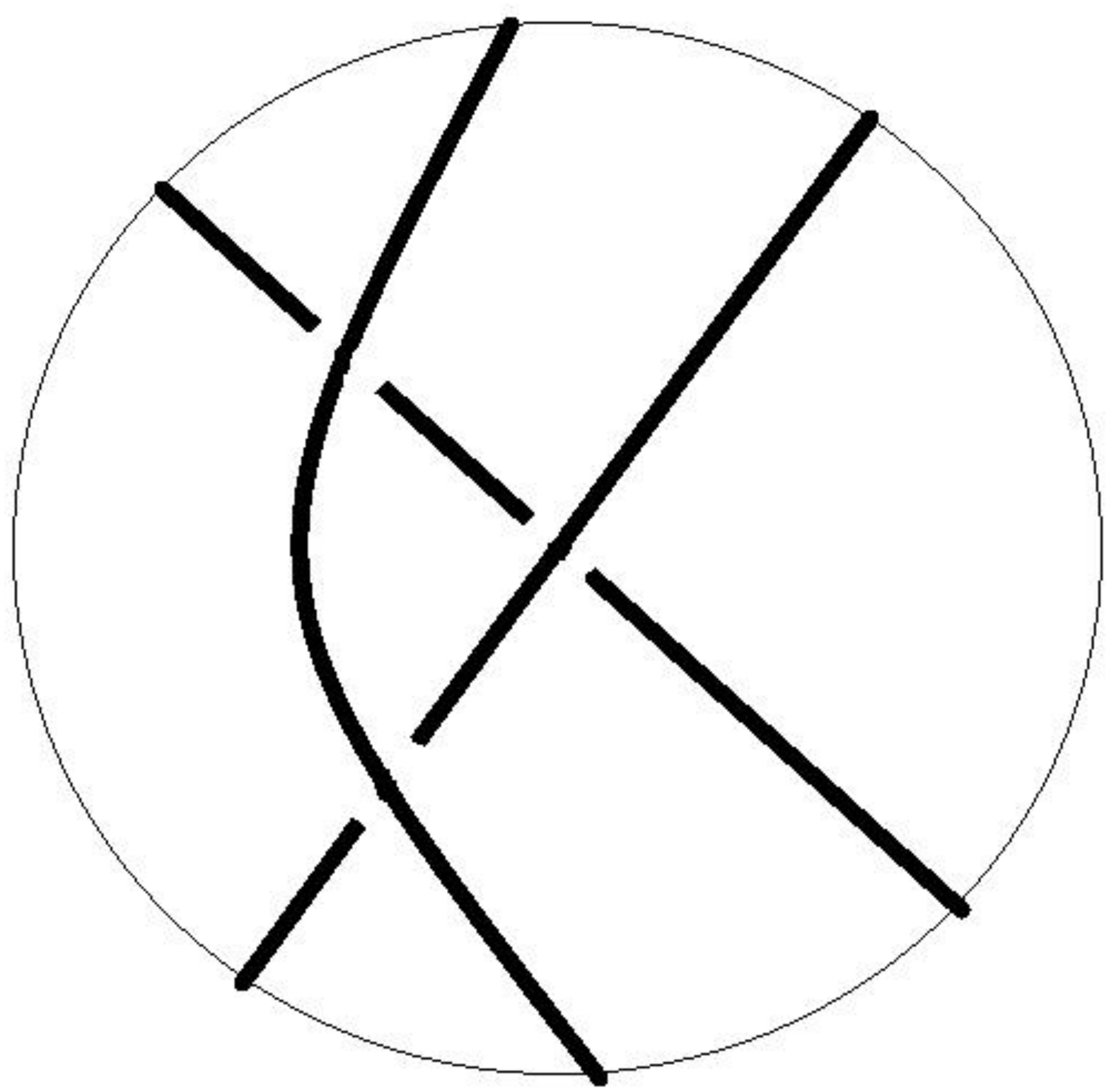}}
\put(60,30){\LARGE{$\leftrightarrow$}}
\end{picture}
\caption{RIII: Reidemeister move III}
\label{fig:RIII}

\end{minipage}
\end{tabular}
\end{figure}

Let $J$ and $J'$ be two finite subsets of $\partial \Sigma$
with $J \cap J' =\emptyset$.
Here $e_1$ and $e_2$ denote the embedding maps from $\Sigma \times I$ to
$\Sigma \times I$ defined by $e_1(x,t) =(x,\frac{t+1}{2})$ and
$e_2(x,t)=(x,\frac{t}{2})$. We define $\boxtimes : \mathcal{T}(\Sigma, J)
\times \mathcal{T}(\Sigma, J') \to \mathcal{T}(\Sigma, J \cup J')$ by
\begin{equation*}
\kukakko{E} \boxtimes \kukakko{E'} \defeq \kukakko{e_1 \circ E \sqcup e_2 \circ E'}
\end{equation*}
for $E \in \mathcal{E}(\Sigma, J)$ and $E' \in \mathcal{E}(\Sigma, J')$.

Let $J$ be a finite subset of $\partial \Sigma$,
$T$ an element of $\mathcal{T}(\Sigma,J)$ 
represented by $E \in \mathcal{E}(\Sigma,J)$ and 
$\xi$ an element of $\mathcal{M}(\Sigma)$ represented by 
a diffeomorphism $\mathcal{X}_\xi$,
where we denote by $\mathcal{M}(\Sigma)$ the mapping class group of 
$\Sigma$ preserving the boundary.
We denote by $\xi T$ an element of 
$\mathcal{T}(\Sigma,J)$ represented by
$(\mathcal{X}_\xi \times \id_I)\circ E \in \mathcal{E}(\Sigma,J)$.

\section{Kauffman bracket skein modules}

\label{section_def_skein}

Throughout this section, 
let $\Sigma$ be a compact connected oriented surface.

\subsection{Definition of Kauffman bracket skein modules}
\label{subsection_define_skein_module}
In this subsection, we define Kauffman bracket skein modules.

First of all, we define a Kauffman triple.

\begin{df}
Let $J$ be a finite subset of $\partial \Sigma$.
A triple of three tangles $T_1$, $T_\infty$ and $T_0 \in \mathcal{T}(\Sigma,J)$ 
is a Kauffman triple if there exist $E_1$, $E_\infty$ and $E_0 \in 
\mathcal{E}(\Sigma,J)$ whose domains are $D_1$, $D_\infty$ and $D_0$
satisfying the following two conditions.
\begin{itemize}
\item We have $\kukakko{E_1} =T_1$, $\kukakko{E_\infty}=T_\infty$ and $\kukakko{E_0}
=T_0$.
\item The three images $E_1(D_1)$, $E_\infty (D_\infty)$ and $E_0 (D_0)$ are identical
except for some neighborhood of a point, where they differ as shown in the figure.

\begin{picture}(300,90)
\put(0,0){\includegraphics[width=60pt]{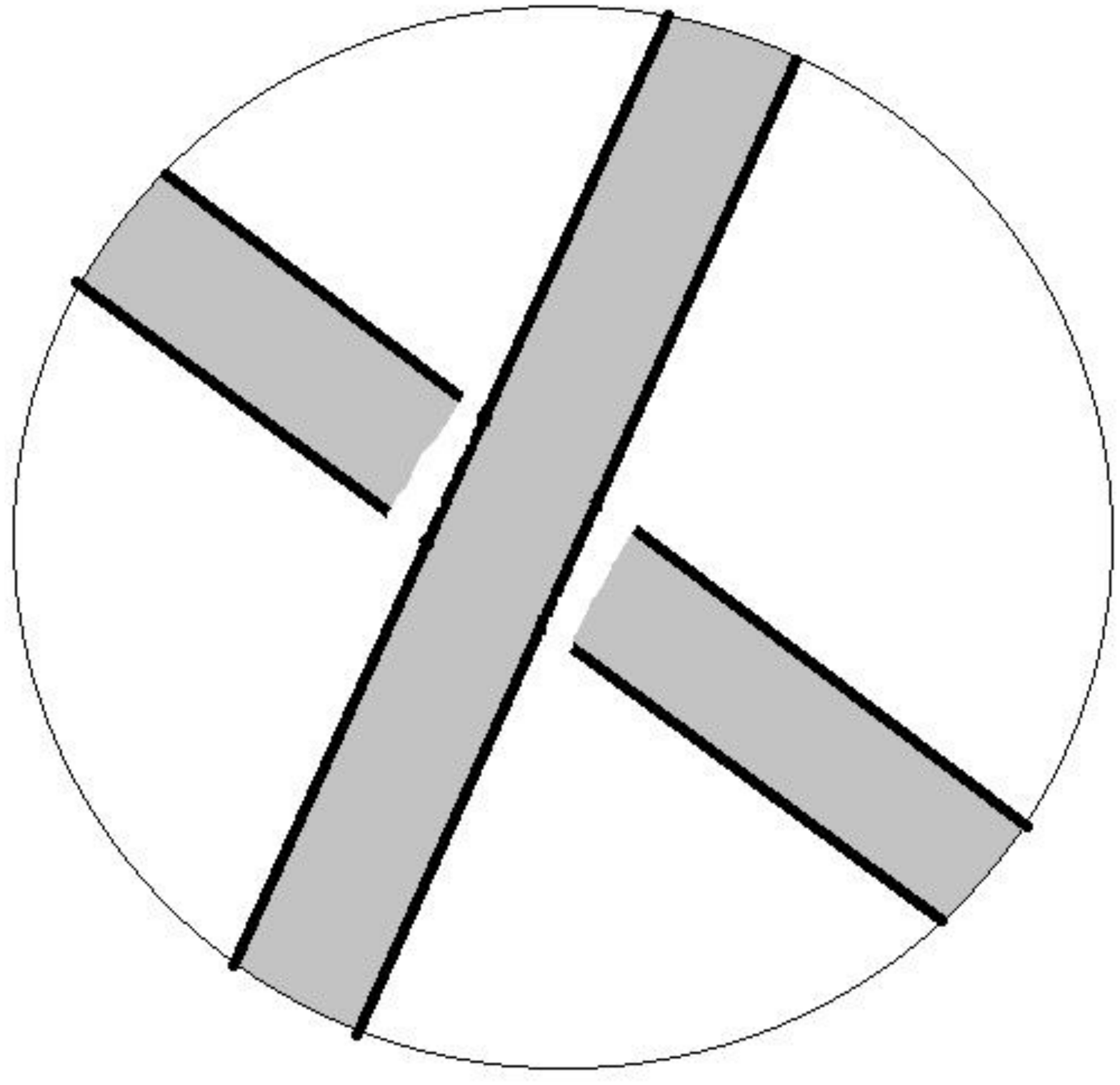}}
\put(70,0){\includegraphics[width=60pt]{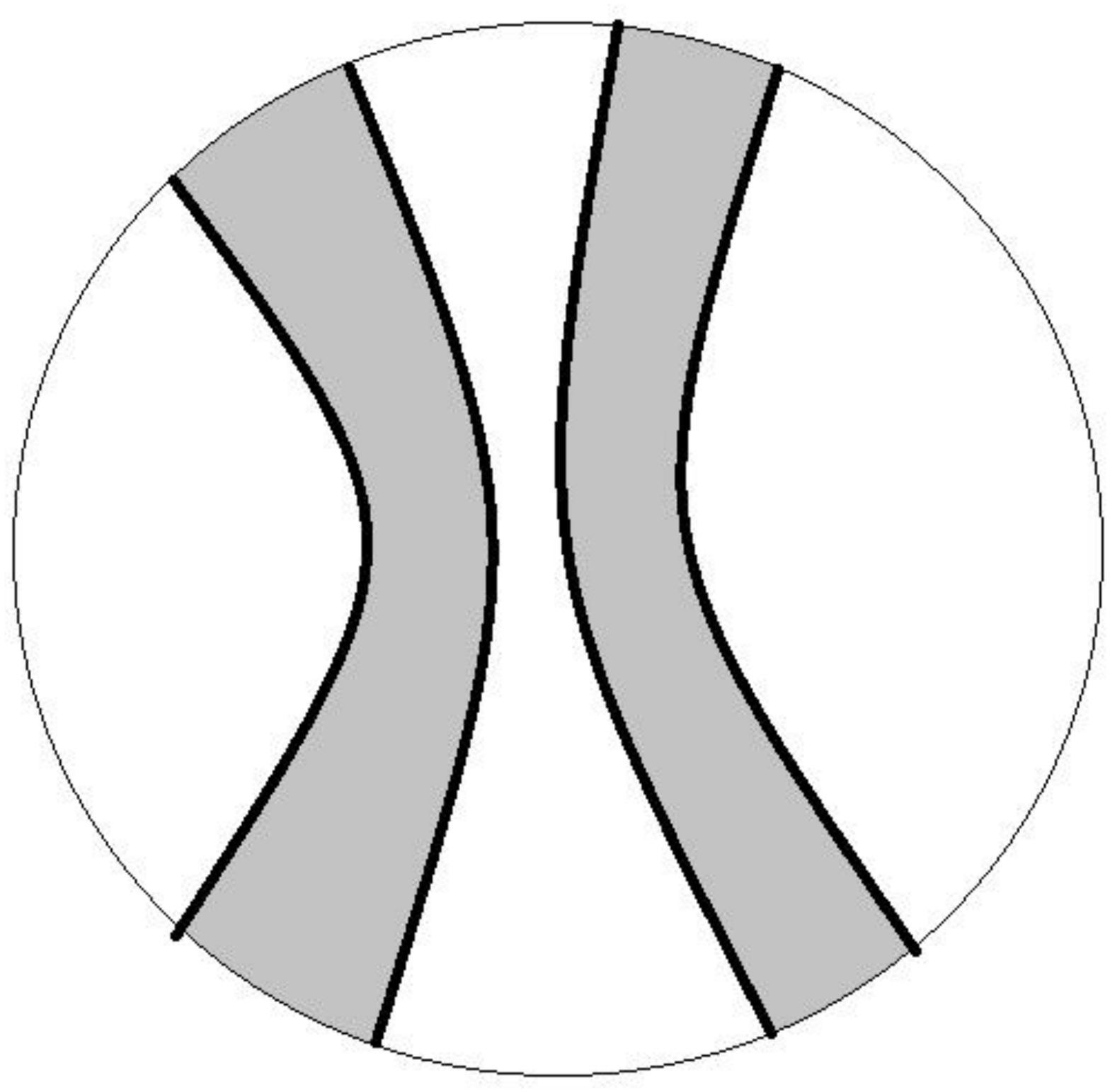}}
\put(140,0){\includegraphics[width=60pt]{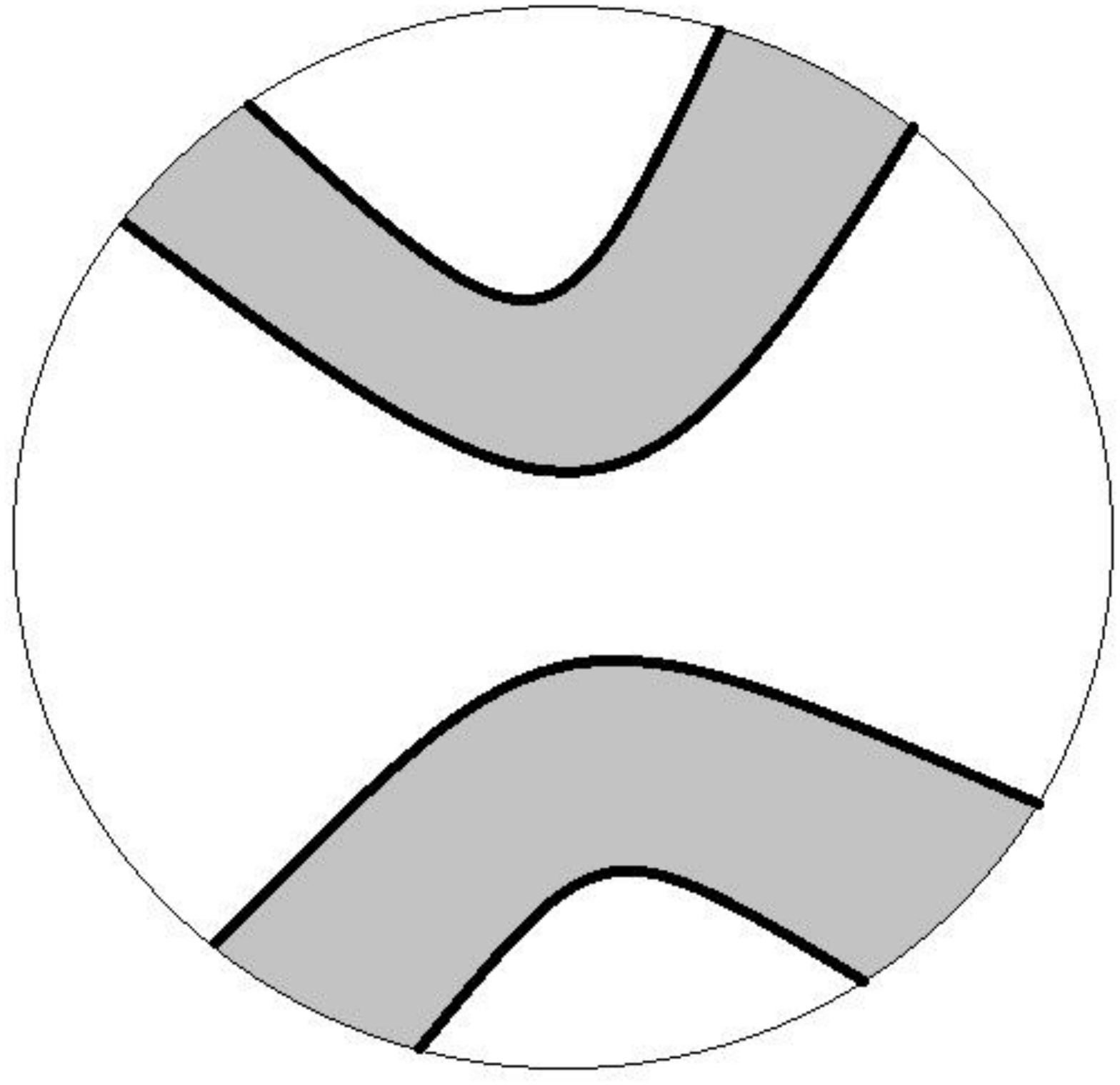}}
\put(10,76){$T_1 :E_1(D_1)$}
\put(80,76){$T_\infty :E_\infty (D_\infty)$}
\put(150,76){$T_0 :E_0(D_0)$}
\end{picture}
\end{itemize}

In other words, there exist three tangle diagrams $d_1$, $d_\infty$ and $d_0$
which present $T_1$, $T_\infty$ and $T_0$ respectively such that
$d_1$, $d_\infty$ and $d_0$ are identical
except for some neighborhood of a point, where they differ as shown in Figure 
\ref{fig_K2}, Figure \ref{fig_Kinfi} and Figure \ref{fig_K0} respectively.

\begin{figure}[htbp]
	\begin{tabular}{rrrr}
	\begin{minipage}{0.25\hsize}
		\centering
		\includegraphics[width=2cm]{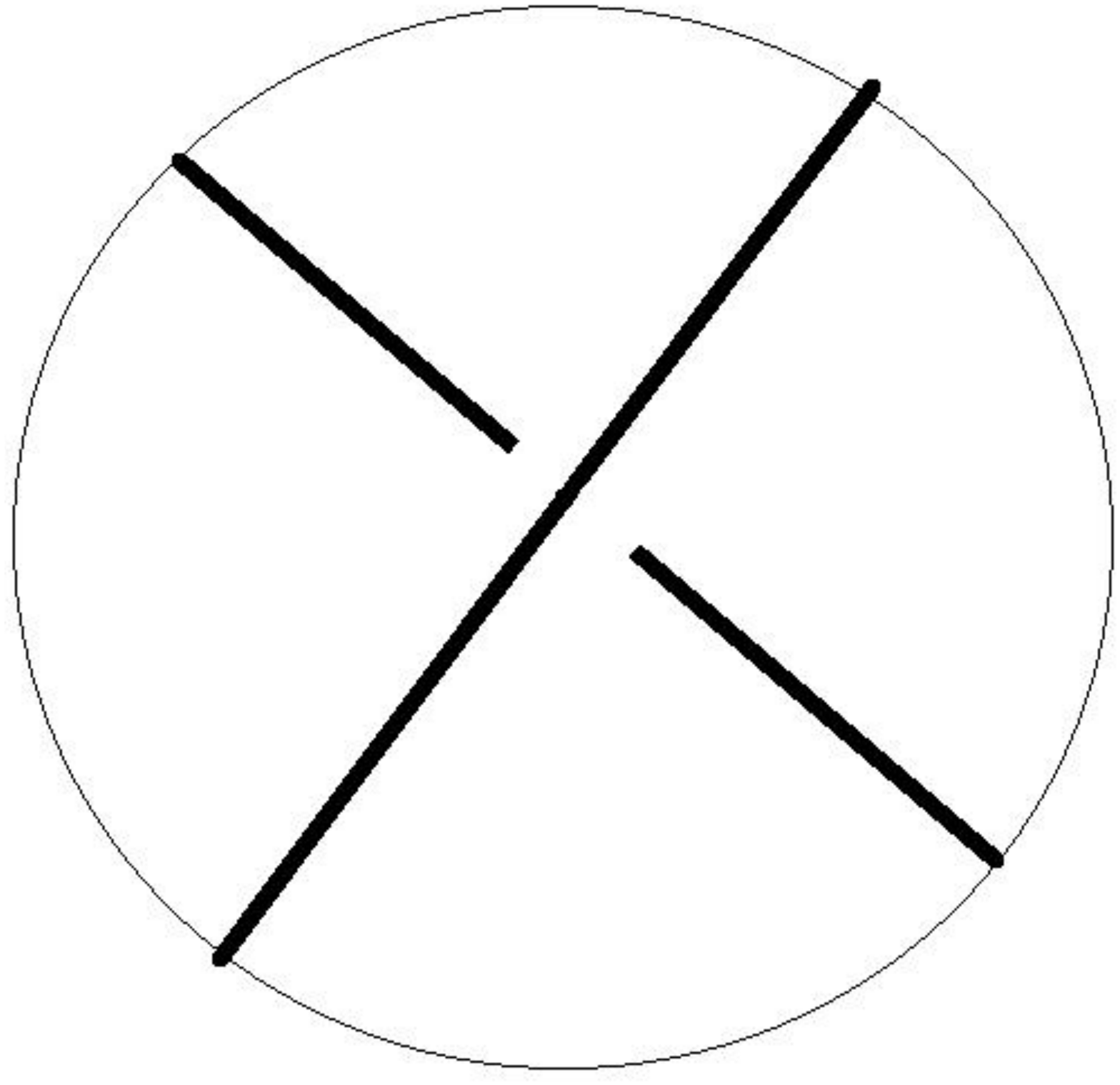}
		\caption{}
		\label{fig_K2}
	\end{minipage}
     \begin{minipage}{0.25\hsize}
		\centering
		\includegraphics[width=2cm]{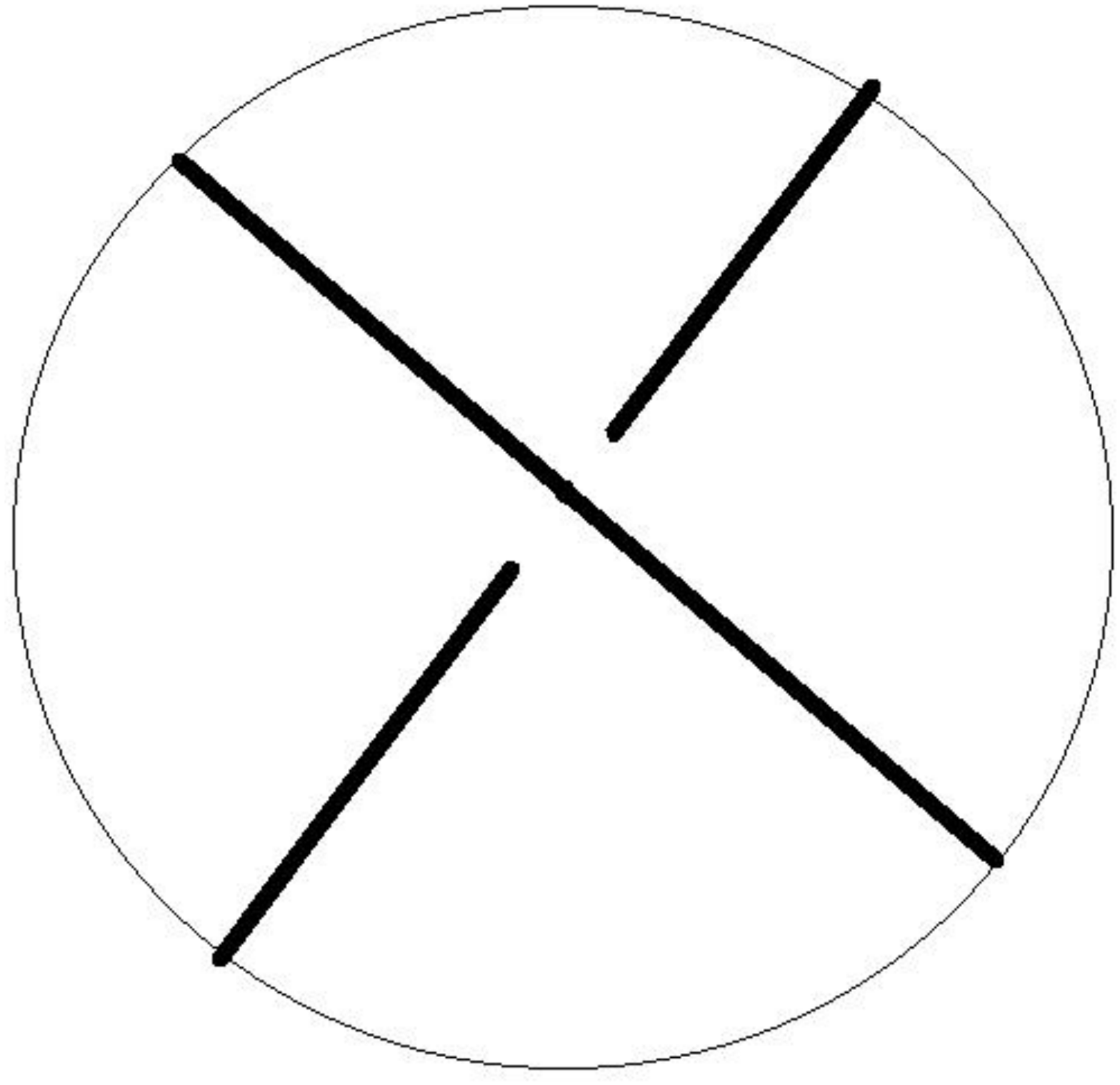}
		\caption{}
		\label{fig_K1}
	\end{minipage}
		\begin{minipage}{0.25\hsize}
		\centering
		\includegraphics[width=2cm]{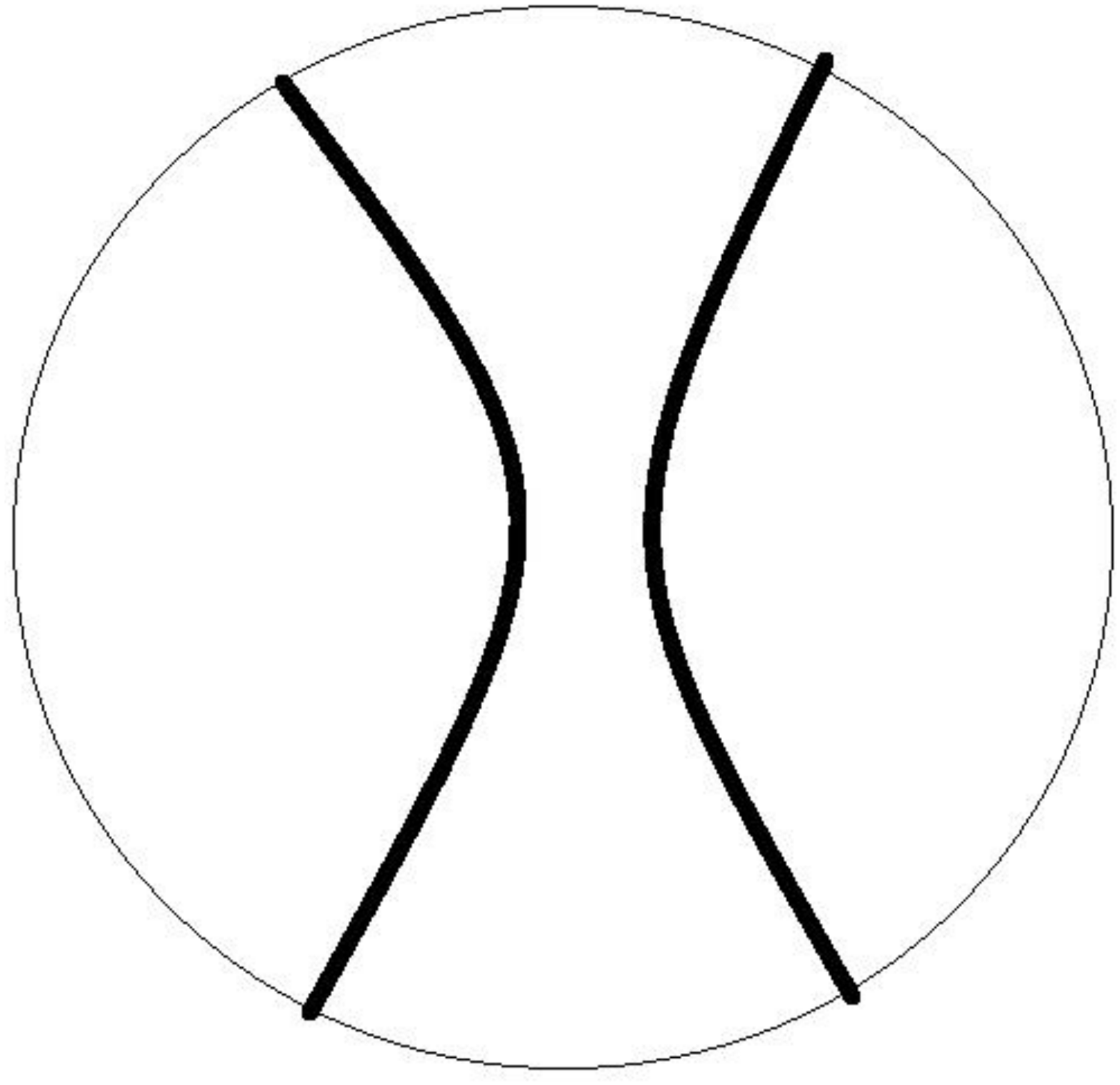}
		\caption{}
		\label{fig_Kinfi}
	\end{minipage}
		\begin{minipage}{0.25\hsize}
		\centering
		\includegraphics[width=2cm]{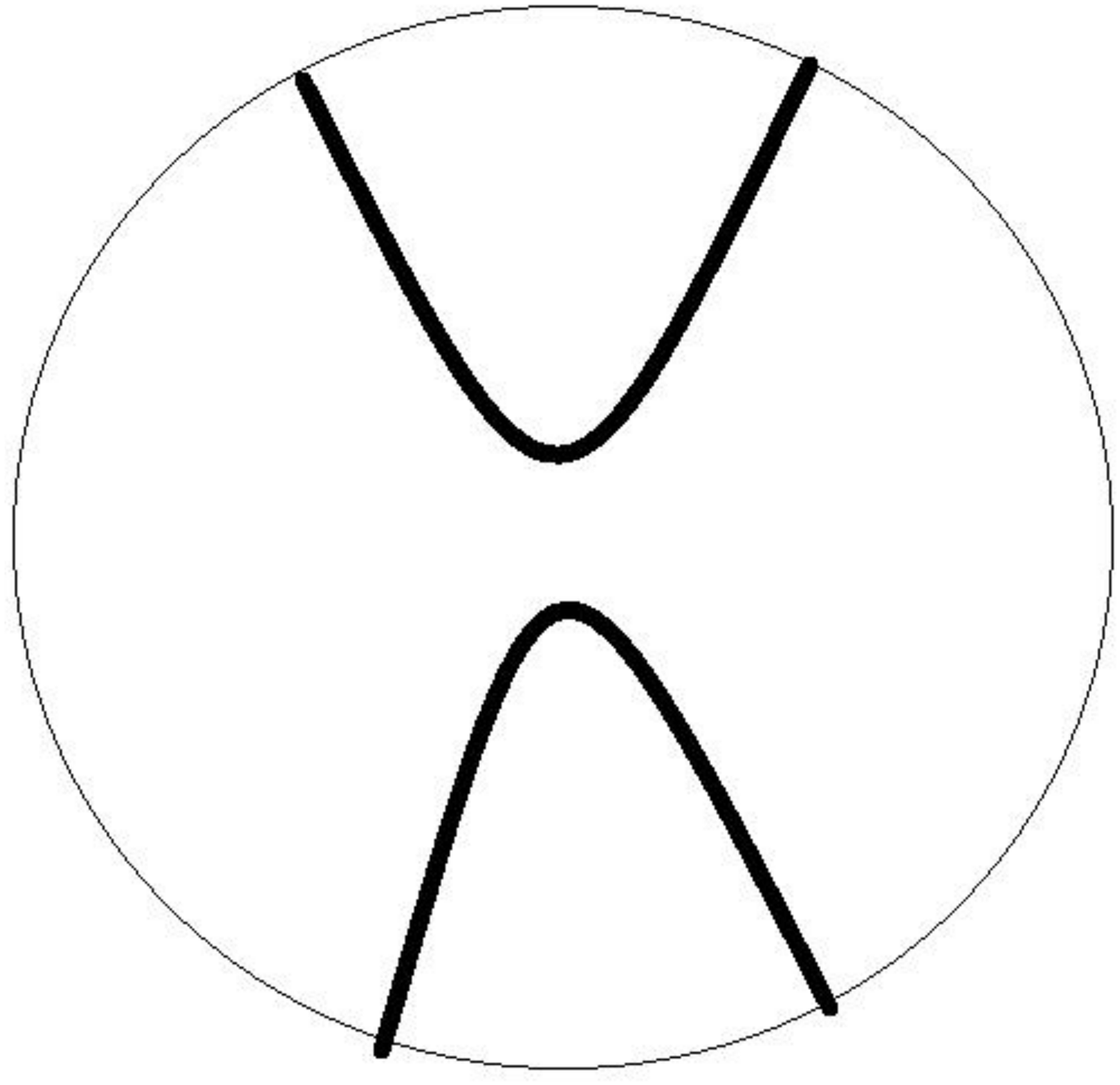}
		\caption{}
		\label{fig_K0}
	\end{minipage}
	\end{tabular}
\end{figure}

\end{df}

We define Kauffman bracket skein modules.

\begin{df}[Kauffman bracket skein module]
Let $J$ be a finite subset of $\partial \Sigma$.
We define  $\skein{\Sigma,J}$  to be the quotient of 
the free module $\Q [A,\gyaku{A}] \mathcal{T} (
\Sigma,J)$ by the skein relation, i.e., 
by the submodule of $\Q [A,\gyaku{A}] \mathcal{T} (
\Sigma,J)$ generated by 
\begin{equation*}
\shuugou{-T_1 +A T_\infty+\gyaku{A} T_0| (T_1,T_\infty,T_0)
\textit{is a Kauffman triple}} \cup \shuugou{T \boxtimes 
\mathcal{O} +(A^2+A^{-2})T|T \in
\mathcal{T}(\Sigma,J)}
\end{equation*}
 where $\mathcal{O} \in \mathcal{T}(\Sigma)$ is the trivial knot.
Following \cite{Turaev},
the element of $\skein{\Sigma,J}$ represented by $T \in \mathcal{T}(\Sigma,J)$
is denoted by $[T]$.
We simply denote $\skein{\Sigma, \emptyset}$ by $\skein{\Sigma}$.
\end{df}

In \cite{Mu2012}, Muller also defined the skein modules for a surface with boundary.
We, however, do not need
`the boundary skein relation' and `the value of a contractible arc'.

Let $J$ and $J'$ be two finite subset of $\partial \Sigma$ 
satisfying $J \cap J' =\emptyset$.
The $\Q[A,\gyaku{A}]$ bilinear homomorphism $\boxtimes :
\skein{\Sigma,J} \times \skein{\Sigma,J'} \to \skein{\Sigma,J \cup J'}$ is defined by
$[T] \boxtimes [T'] \defeq [T \boxtimes T']$ for $T \in \mathcal{T} (\Sigma,J)$ and
$T' \in \mathcal{T} (\Sigma,J')$.
The skein module $\skein{\Sigma}$ is the associative algebra over
 $\Q[A,\gyaku{A}]$ with product
 defined by $ab =a \boxtimes b$ for $a$ and $b \in \skein{\Sigma}$.
The skein module $\skein{\Sigma,J}$ is the $\skein{\Sigma}$-bimodule given
by $av =a \boxtimes v$ and
$va =v \boxtimes a$ for $a \in \skein{\Sigma}$ and $v \in \skein{\Sigma,J}$.
For $v \in \skein{\Sigma,J}$, $v' \in \skein{\Sigma,J'}$ and $a \in \skein{\Sigma}$,
we have $(va)\boxtimes v' =v \boxtimes (av')$.

\subsection{Some Poisson-like structure on $\skein{\Sigma}$}
\label{subsection_Poisson_structure}
In this subsection, we define a Lie bracket of $\skein{\Sigma}$ and
an action $\sigma$ of $\skein{\Sigma}$ on $\skein{\Sigma,J}$.

Let $J$ be a finite subset of $\partial \Sigma$. We denote by $\mathcal{E}_0(\Sigma,J)$
the set of $1$-dimensional submanifolds of
$\Sigma$ with boundary $J$ and no inessential components.
Here a  connected $1$-dimensional submanifold of $\Sigma$ is inessential
if it is a boundary of a disk in $\Sigma$.
We denote the set of isotopy classes in $\mathcal{E}_0(\Sigma,J)$ by
 $\mathcal{T}_0(\Sigma,J)$.

\begin{thm}
\label{thm_przy}
Let $J$ be a finite subset of $\partial \Sigma$.
The skein module $\skein{\Sigma,J}$ is the $\Q[A,A^{-1}]$-free module with basis 
$\mathcal{T}_0(\Sigma,J)$.
\end{thm}

In the case when $J=\emptyset$,
this is proved by
Przytycky \cite{skeinmodule}.
For the general case, it is proved in a similar way to \cite{skeinmodule}.

\begin{cor}
We have $\skein{S^1 \times I} =\Q[A,\gyaku{A}] [l]$ where
$l$ is the element represented by the link whose diagram is
$S^1 \times \shuugou{\frac{1}{2}}$.
\end{cor}

\begin{cor}
Let $J$ be a finite subset of $\partial \Sigma$. 
The $\Q [A,\gyaku{A}]$-module homomorphism 
$-A+\gyaku{A} :\skein{\Sigma,J} \to \skein{\Sigma,J},x \mapsto (-A+\gyaku{A})x$
is an injective map.
\end{cor}

\begin{lemm}
\label{lemm_kouten_sa}
Let $J$ be a finite subset of $\partial \Sigma$.
Let $T_1$, $T_2$, $T_3$ and $T_4$ be four elements of 
$\mathcal{T}(\Sigma,J)$ presented by four diagrams
which are identical
except for some neighborhood of a point, 
where they differ as shown in Figure \ref{fig_K2},
Figure \ref{fig_K1}, Figure \ref{fig_Kinfi} and Figure \ref{fig_K0}
respectively.
Then we have $[T_1]-[T_2] =(A-\gyaku{A})([T_3]-[T_4])$.

\end{lemm}

\begin{proof}
We have 
\begin{equation*}
[T_1]-[T_2]=(A[T_3]+A^{-1}[T_4])-(\gyaku{A}[T_3]+A[T_4])=
(A-\gyaku{A})([T_3]-[T_4]).
\end{equation*}
\end{proof}

In Definition \ref{definition_Lie_bracket},
we introduce a Lie bracket in $\skein{\Sigma}$
by using the following proposition.

\begin{prop}
\label{tensor_sa}
Let $J$ and $J'$ be two finite subsets of $\partial \Sigma$ satisfying $J \cap J' =\emptyset$. We have $v \boxtimes v' - v' \boxtimes v \in (A-\gyaku{A})
\skein{\Sigma, J\cup J'}$ for $v \in \skein{\Sigma,J}$ and $v' \in \skein{\Sigma,J'}$.

\end{prop}

\begin{proof}
Let $T$ be an element of $\mathcal{T}(\Sigma,J)$ and
$T'$ an element of $\mathcal{T}(\Sigma,J')$.
Choose tangle diagrams $d$ and $d'$ presenting $T$ and $T'$, respectively, 
such that the intersections of $d$ and $d'$ consist of transverse double points
$P_1, P_2, \cdots, P_m$.
For $i=1,2, \cdots, m$, let $d(1,i)$ and $d(-1,i)$ be two tangle diagrams
satisfying the following contitions.
\begin{itemize}
\item The two tangle diagrams $d(1,i)$ and $d(-1,i)$ equal
$d \cup d'$ with the same height-information as
$d$ and $d'$ except for the neighborhoods of the intersections
of $d$ and $d'$.
\item The branches of $d(1,i)$ and $d(-1,i)$ in the neighborhood of $P_j$
belonging to $d'$ are overcrossings for
$j=1, \cdots, i-1$.
\item The branches of $d(1,i)$ and $d(-1,i)$ in the neighborhood of $P_j$
belonging to $d$ are overcrossings for 
$j=i+1, \cdots,m$.
\item The two tangle diagrams $d(1,i)$ and $d(-1,i)$  are as shown in Figure
\ref{fig_d_and_d_1}  and Figure \ref{fig_d_and_d_2}, respectively, 
in the neighborhood of $P_i$.

\end{itemize}

\begin{figure}[htbp]
	\begin{tabular}{cc}
	\begin{minipage}{0.33\hsize}
		\centering
		\includegraphics[width=3cm]{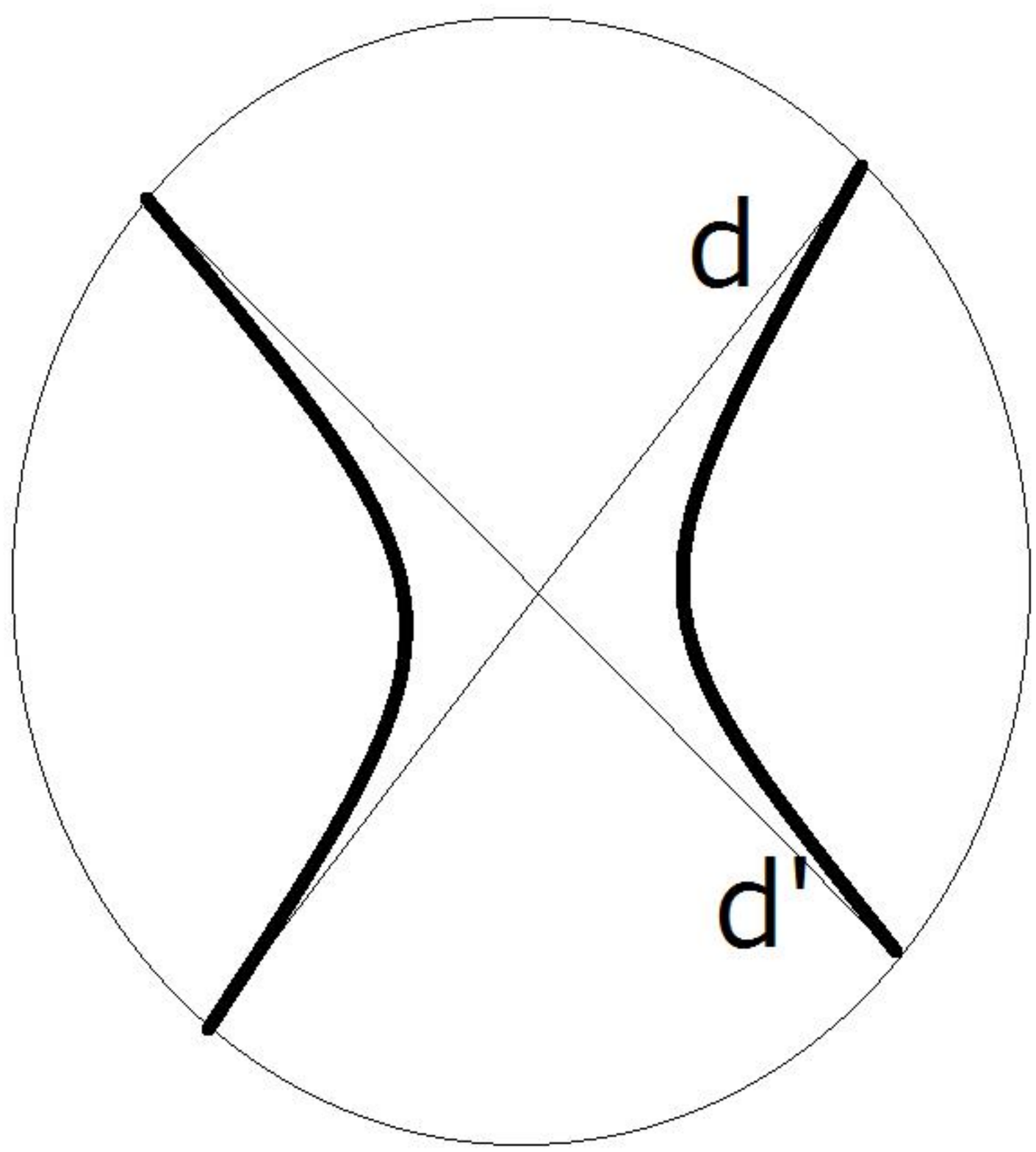}
		\caption{}
		\label{fig_d_and_d_1}
	\end{minipage}
     \begin{minipage}{0.33\hsize}
		\centering
		\includegraphics[width=3cm]{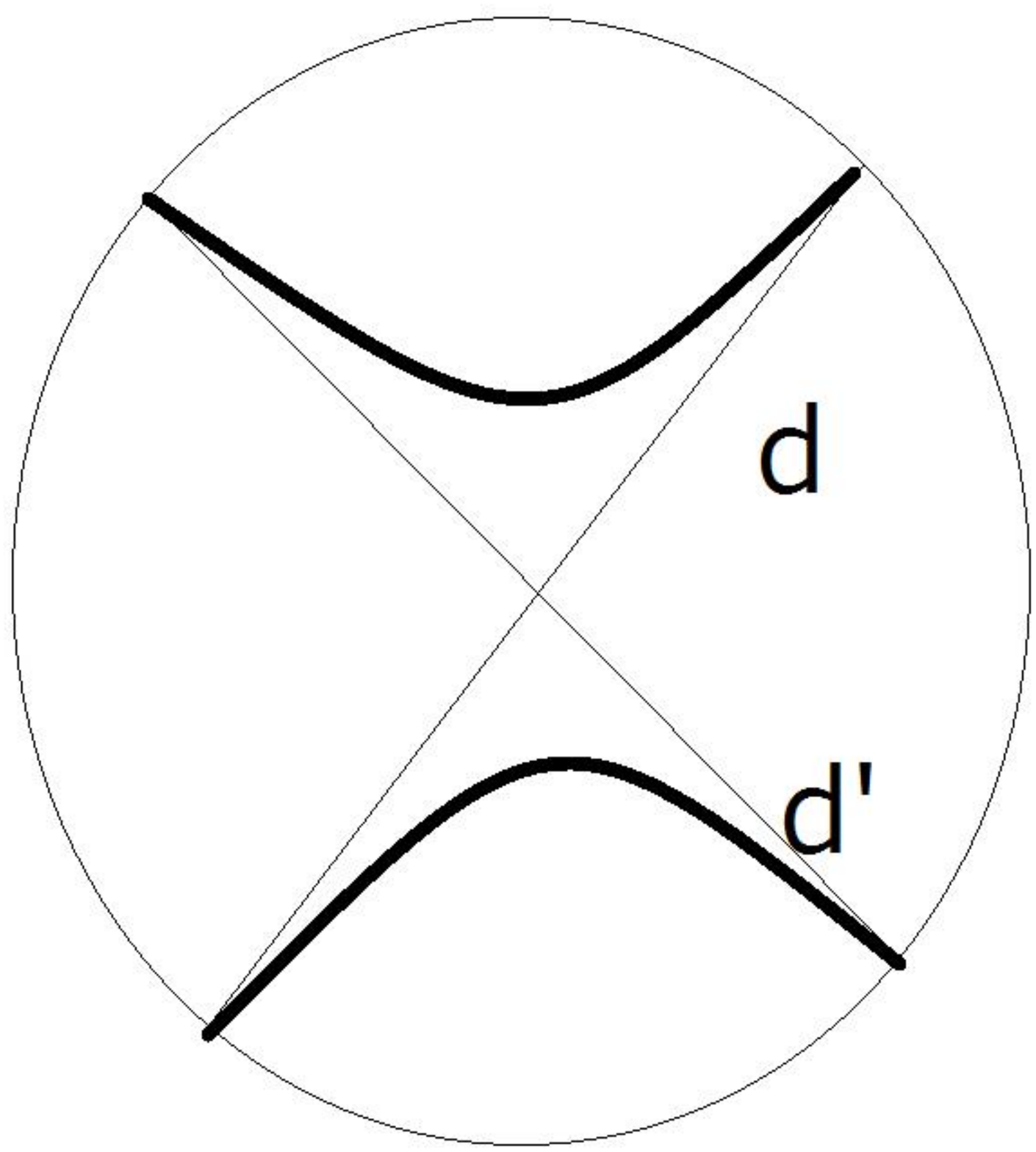}
		\caption{}
		\label{fig_d_and_d_2}
	\end{minipage}
	\end{tabular}
\end{figure}
We denote by $T(1,i)$ a tangle presented by $d(1,i)$ and by $T(-1,i)$ a 
tangle presented by $d(-1,i)$. 
Using lemma \ref{lemm_kouten_sa}, we have 
\begin{equation*}
[T]\boxtimes [T']-[T']\boxtimes[T] =(A-\gyaku{A}) \sum^m_{i=1}([T(1,i)]-[T(-1,i)]).
\end{equation*}

This proves the proposition.

\end{proof}

\begin{df}\label{definition_Lie_bracket}
Let $J$ be a finite subset of $\partial \Sigma$.
We define a bracket $[,]$ of $\skein{\Sigma}$ by
\begin{equation*}
[x,y] \defeq \frac{1}{-A+\gyaku{A}}(xy-yx)
\end{equation*}
for $x$ and $y \in \skein{\Sigma}$.
We define an action $\sigma$ of $\skein{\Sigma}$ on $\skein{\Sigma,J}$ by
\begin{equation*}
\sigma(x)(v) \defeq \frac{1}{-A+\gyaku{A}}(xv-vx)
\end{equation*}
for $x \in \skein{\Sigma}$ and $v \in \skein{\Sigma,J}$.
\end{df}

It is easy to prove the following proposition.

\begin{prop}
Let $J$ be a finite subset of $\partial \Sigma$. 
The bracket $[,]:\skein{\Sigma} \times \skein{\Sigma} \to \skein{\Sigma}$ makes
$\skein{\Sigma}$ a Lie algebra. The action $\sigma : \skein{\Sigma} \times \skein{\Sigma
,J} \to \skein{\Sigma,J}$ makes $\skein{\Sigma,J}$ a $\skein{\Sigma}$-module when we 
regard $\skein{\Sigma}$ as a Lie algebra.
Furthermore, for $x$, $y$ and $z \in \skein{\Sigma}$ and $v \in \skein{\Sigma,J}$, 
we have the Leibniz rules:
\begin{align*}
&[xy,z]=x[y,z]+[x,z]y, \\
&\sigma(xy)(v)=x\sigma(y)(v)+\sigma(x)(v)y, \\
&\sigma(x)(yv)=[x,y]v +y \sigma(x)(v), \\
&\sigma(x)(vy)=\sigma(x)(v)y+ v[x,y].
\end{align*}
\end{prop}

Let $J$ and $J'$ be two finite subsets of $\partial \Sigma$.
We have 
\begin{equation*}
\sigma(x)(v \boxtimes v')=\sigma(x)(v)\boxtimes v' +v \boxtimes \sigma(x)(v')
\end{equation*}
for $x \in \skein{\Sigma}$, $v \in \skein{\Sigma,J}$ and $v' \in \skein{\Sigma,J'}$.

\subsection{Filtrations and completions}
\label{subsection_filtrations_and_completions}
We introduce filtrations of Kauffman bracket skein modules and define 
the completed Kauffman bracket skein modules.

We define an augmentation map $\epsilon : \skein{\Sigma} \to \Q$
by $A \mapsto -1$ and $[L] \mapsto (-2)^{\zettaiti{L}}$
for $L \in \mathcal{T} (\Sigma)$ where $\zettaiti{L}$ is the number of 
components of $L$. 

\begin{prop}
The augmentation map $\epsilon$ is well-defined.
\end{prop}

\begin{proof}
Let $T_1$, $T_\infty$ and $T_0$ be three elements of 
$\mathcal{T} (\Sigma)$ such that $(T_1,T_\infty, T_0)$ is a
Kauffman triple. There are three cases,
\begin{align*}
&\zettaiti{T_1}-1 = \zettaiti{T_\infty} = \zettaiti{T_0} ,\\
&\zettaiti{T_1} = \zettaiti{T_\infty}-1 = \zettaiti{T_0} ,\\
&\zettaiti{T_1} = \zettaiti{T_\infty} = \zettaiti{T_0}-1 .
\end{align*}
In each case, we have $\epsilon ([T_1]-A[T_\infty]-\gyaku{A}[T_0]) =0$.
For $T \in \mathcal{T}(\Sigma)$, we have $\epsilon ([T \boxtimes
 \mathcal{O}]+(A^2+A^{-2})T]) =0$.
This proves the lemma.
\end{proof}

\begin{lemm}
\label{lemm_skein_kernel}
We have $[\skein{\Sigma},\skein{\Sigma}] \subset \ker \epsilon$.
\end{lemm}

\begin{proof}
Since $\skein{\Sigma}$ is generated by the sets 
of elements represented by knots, it suffices to show that $[[T],[T']] \in \ker \epsilon$
for any two elements $T$ and $T'$ of $\mathcal{T} (\Sigma)$
satisfying $\zettaiti{T} =1$ and $\zettaiti{T'}=1$.
Choose tangle diagrams $d$ and $d'$ presenting $T$ and $T'$, respectively, 
such that the intersections of $d$ and $d'$ consist of transverse double points
$P_1, P_2, \cdots, P_m$.
For $i=1,2, \cdots, m$, let $d(1,i)$ and $d(-1,i)$ be two tangle diagrams
satisfying the following contitions.
\begin{itemize}
\item The two tangle diagrams $d(1,i)$ and $d(-1,i)$ equal
$d \cup d'$ with the same height-information as
$d$ and $d'$ except for the neighborhoods of the intersections
of $d$ and $d'$.
\item The branches of $d(1,i)$ and $d(-1,i)$ in the neighborhood of $P_j$
belonging to $d'$ are over crossings for
$j=1, \cdots, i-1$.
\item The branches of $d(1,i)$ and $d(-1,i)$ in the neighborhood of $P_j$
belonging to $d$ are over crossings for 
$j=i+1, \cdots,m$.
\item The two tangle diagrams $d$ and $d'$  are as shown in Figure
\ref{fig_d_and_d_1}  and Figure \ref{fig_d_and_d_2}, respectively, 
in the neighborhood of $P_i$.

\end{itemize}
We denote by $T(1,i)$ a tangle presented by $d(1,i)$ and by $T(-1,i)$ a 
tangle presented by $d(-1,i)$. 
Using lemma \ref{lemm_kouten_sa}, we have 
\begin{equation*}
[[T],[T']] =-\sum^m_{i=1}([T(1,i)]-[T(-1,i)]).
\end{equation*}
We remark  that $T(1,i)$ and $T(-1,i)$ are knots
for $i  \in \shuugou{1, \cdots, m}$.
We have $\epsilon(-\sum^m_{i=1}([T(1,i)]-[T(-1,i)]))
=0$.
This proves the proposition.

\end{proof}

Let $J$ be a finite subset of $\partial \Sigma$.
We define a filtration of $\skein{\Sigma}$ by
$F^n \skein{\Sigma}=(\ker \epsilon)^n$
and a filtration of $\skein{\Sigma,J}$
by $F^n \skein{\Sigma,J} =(F^n \skein{\Sigma}) \skein{\Sigma,J}$.

\begin{thm}
(1) Let $J$ be a finite subset of $\partial \Sigma$.
We have 
\begin{align*}
& F^n \skein{\Sigma} F^m \skein{\Sigma}  \subset F^{n+m} \skein{\Sigma}, \\
& F^n \skein{\Sigma}F^m \skein{\Sigma,J} \subset F^{n+m} \skein{\Sigma,J}, \\
& F^n \skein{\Sigma,J} F^m \skein{\Sigma} \subset F^{n+m} \skein{\Sigma,J},
\end{align*}
for $n$ and $m \in \Z_{\geq 0}$.

(2) we have $[F^i \skein{\Sigma},F^j \skein{\Sigma}] \subset
F^{\mathrm{max}(i+j-1,i,j)}\skein{\Sigma}$ and $\sigma (F^i \skein{\Sigma})(F^j \skein{\Sigma,J})
\subset F^{\mathrm{max}(i+j-1,i-1,j)}\skein{\Sigma,J}$ for $i$ and $j \in \Z{\geq 0}$.
\end{thm}

\begin{proof}
In order to show, for $i$ and $j \in \Z_{\geq 0}$,
\begin{align*}
F^i \skein{\Sigma} F^j \skein{\Sigma,J} \subset F^{i+j} \skein{\Sigma,J}, \\
F^j \skein{\Sigma,J} F^i \skein{\Sigma} \subset F^{i+j} \skein{\Sigma,J},
\end{align*}
it suffices to prove
\begin{equation*}
(\ker \epsilon) \skein{\Sigma,J} = \skein{\Sigma,J} (\ker \epsilon),
\end{equation*}
which is obvious by Proposition \ref{tensor_sa}.
This proves (1).

Using the Leibniz rule,
Lemma \ref{lemm_skein_kernel} show, for $i$ and $j \in \Z{\geq 0}$, 
\begin{align*}
[F^i \skein{\Sigma},F^j \skein{\Sigma}] &\subset
F^{\mathrm{max}(i+j-1,i,j)} \skein{\Sigma}, \\
\sigma (F^i \skein{\Sigma})(F^j \skein{\Sigma,J})
&\subset F^{\mathrm{max}(i+j-1,i-1,j)}\skein{\Sigma,J}.
\end{align*}
This proves (2).

\end{proof}

Let $J$ be a finite subset of $\partial \Sigma$.
We define an action of $\mathcal{M}(\Sigma)$ on $\skein{\Sigma,J}$
by $\xi [T] =[\xi T]$ for $\xi \in \mathcal{M}(\Sigma)$ and
$T \in \mathcal{T}(\Sigma,J)$.
We have 
\begin{align*}
&\xi (F^n \skein{\Sigma}) =F^n \skein{\Sigma}, \\ 
&\xi (F^n \skein{\Sigma,J}) = F^m \skein{\Sigma,J} \\
\end{align*}
for $\xi \in \mathcal{M} (\Sigma)$ and
$n \in \Z_{\geq 0}$.

\begin{rem}
We have $\mathrm{dim}_{\Q} (F^n \skein{\Sigma,J}/F^{n+1} \skein{\Sigma,J}) < \infty$.
The proof will appear in \cite{TsujiCSAII}.
\end{rem}

Let $J$ be a finite subset of $\partial \Sigma$.
We consider the topology on $\skein{\Sigma}$ induced by
the filtration $\filtn{F^n \skein{\Sigma}}$, and denote its completion
by $\widehat{\skein{\Sigma}} \defeq
\underleftarrow{\lim}_{i \rightarrow \infty}\skein{\Sigma}/
F^i \skein{\Sigma}$.
We call $\widehat{\skein{\Sigma}}$ the completed skein algebra.
We also consider the topology on $\skein{\Sigma,J}$ induced by
the filtration $\filtn{F^n \skein{\Sigma,J}}$, and denote its completion
by $\widehat{\skein{\Sigma,J}} \defeq
\underleftarrow{\lim}_{i \rightarrow \infty}\skein{\Sigma,J} /
F^i \skein{\Sigma,J}$.
We call $\widehat{\skein{\Sigma,J}}$ the completed skein module.
The completed skein algebra $\widehat{\skein{\Sigma}}$ 
has a filtration $\widehat{\skein{\Sigma}} = F^0 \widehat{\skein{\Sigma}} \supset 
F^1 \widehat{\skein{\Sigma}} \supset F^2 \widehat{\skein{\Sigma}} \supset \cdots$
such that $\widehat{\skein{\Sigma}}/F^n \widehat{\skein{\Sigma}} \simeq
\skein{\Sigma}/F^n \skein{\Sigma}$ for $n \in \Z_{\geq 0}$.
The completed skein module $\widehat{\skein{\Sigma,J}}$ 
also has a filtration $\widehat{\skein{\Sigma,J}} = F^0 \widehat{\skein{\Sigma,J}} \supset 
F^1\widehat{\skein{\Sigma,J}} \supset F^2 \widehat{\skein{\Sigma,J}} \supset \cdots$
such that $\widehat{\skein{\Sigma,J}}/F^n \widehat{\skein{\Sigma,J}} 
\simeq
\skein{\Sigma,J}/F^n \skein{\Sigma,J}$ for $n \in \Z_{\geq 0}$.
We remark that the completed skein algebra $\widehat{\skein{\Sigma}}$ is 
an associative $\Q [[A+1]]$-algebra and that
the completed skein module $\widehat{\skein{\Sigma,J}}$
is a $\Q[[A+1]]$-module.
The set $\shuugou{(\skein{\Sigma,J},\filtn{F^n\skein{\Sigma,J}})|J \subset \partial \Sigma,
\sharp J < \infty}$ is denoted by $\Theta (\Sigma)$.

We denote by $\check{\mathcal{M}} (\Sigma)
\subset  \mathcal{M} (\Sigma)$ the subset 
consisting of  elements $\xi$ satisfying the condition
that, for any finite subset $J$ of $\partial{\Sigma}$, any non-negative integer $m$
and any element $v \in F^m \skein{\Sigma,J}$,
there exists a non-negative integer
$N$ such that $j \geq N \Rightarrow (1-\xi)^j 
(v) \in F^{m+1} \skein{\Sigma,J}$.

For $\xi \in \check{\mathcal{M}}(\Sigma)$
 and a finite subset $J$ of $\partial \Sigma$ ,
 a $\Q[[A+1]]$-module homomorphism
$\log(\xi) :\widehat{\skein{\Sigma , J}} \to \widehat{\skein{\Sigma ,J}}$ is defined by
$\log(\xi)(v) = \sum_{i=1}^\infty \frac{-1}{i} (\id - \xi)^i (v)$.
For $\xi \in \check{\mathcal{M}}(\Sigma)$, $x \in \widehat{\skein{\Sigma}}$
 and $z \in \widehat{\skein{\Sigma,J}}$,
since $\xi(xz)=\xi(x)\xi(z)$ and $\xi(zx)=\xi(z)\xi(x)$,
$\log(\xi)$ satisfies the Leibniz rule
\begin{align*}
&\log (\xi)(xz) =\log (\xi)(x)z +x\log (\xi)(z), \\
&\log (\xi)(zx) = \log (\xi)(z)x+z\log(\xi)(x).
\end{align*}

\begin{df}
\label{definition_quantization}
For $\xi \in \check{M}(\Sigma,\partial \Sigma)$, an element $ x_\xi \in 
\widehat{\skein{\Sigma}}$ 
has a skein representative of $\xi$ by $((\skein{\Sigma},\filtn{\skein{\Sigma}}), 
\Theta (\Sigma))$ if we have 
\begin{equation*}
\log (\xi) = \sigma (x_\xi) : \widehat{\skein{\Sigma,J}} \to \widehat{\skein{\Sigma,J}},
\end{equation*}
in other words
\begin{equation*}
\xi (\cdot) =\exp (\sigma(x_\xi)): \widehat{\skein{\Sigma,J}} \to \widehat{\skein{\Sigma,J}},
\end{equation*}
for  a finite subset $J$ of $\partial \Sigma$ .
\end{df}

\section{Dehn twists}
\label{section_Dehn_twists}

In this section we show the following theorem.

\begin{thm}
Let $\Sigma$ be a compact connected oriented surface and $c$
a simple closed curve.
We also denote by $c$ an element of $\skein{\Sigma}$
represented by a knot presented by the simple closed curve $c$.
Then we have $t_c \in \check{\mathcal{M}} (\Sigma)$, and
 $\frac{-A+\gyaku{A}}{4\log(-A)} (\arccosh  (-\frac{c}{2}))^2 \in
\widehat{\skein{\Sigma}}$ 
has a skein representative of $t_c \in \check{\mathcal{M}}(\Sigma
)$ by $((\skein{\Sigma},\filtn{F^n\skein{\Sigma}}), \Theta (\Sigma))$
in the sense of Definition \ref{definition_quantization}.
Here $\frac{-A+A^{-1}}{4 \log (-A)}$ is an element of $ \Q [[A+1]]$
and  $(\arccosh (-\frac{c}{2}))^2$
is $ \sum_{i=0}^\infty \frac{i! i!}{(i+1) (2i+1)!} (1-\frac{c^2}{4})^{i+1} \in
\Q  [[c+2]]$.
\end{thm}

We denote by $S^1 \defeq \R/\Z$,
by $c_l$ a simple closed curve $S^1 \times \shuugou{\frac{1}{2}}$ in
$S^1 \times I$, by $t$ the Dehn twist along $c_l$
and by $l$ an element of $\skein{S^1 \times I}$ represented by
the knot presented by $c_l$.
We fix an embedding $\iota :S^1 \times I \to \Sigma$ such that 
$\iota (c_l) =c$.

We assume that $\iota (S^1 \times I)$ separate $\Sigma$ into two surfaces
$\Sigma^1 $ and $\Sigma^2$.
For a finite set $J' \in S^1$, we consider the trilinear map
\begin{align*}
\varpi_{J'} : \skein{\Sigma^1, (J \cap \partial \Sigma^1) \cup \iota (J'  \times \shuugou{1})}
\times  \skein{S^1 \times I,J'\times \shuugou{0,1}} \\
\times \skein{\Sigma^2,(J \cap \partial \Sigma^1) \cup \iota (J' \times \shuugou{0})} 
\to \skein{\Sigma,J}
\end{align*}
defined by $\varpi_{J'} ([T_1],[T_2],[T_3]) =[T_1T_2T_3]$ for 
$T_1 \in \mathcal{T}(\Sigma^1, (J \cap \partial \Sigma^1) \cup \iota (J'  \times \shuugou{1}))$,
$T_2 \in \mathcal{T}(S^1 \times I,J'\times \shuugou{0,1})$ 
and $T_3 \in \mathcal{T} (
\Sigma^2,(J \cap \partial \Sigma^1) \cup \iota (J' \times \shuugou{0}))$.
Here we denote by $T_1T_2T_3$ the tangle presented by
$d_1 \cup \iota (d_2) \cup d_3$, respectively, where 
$d_1$, $d_2$ and $d_3$ present $T_1$, $T_2$ and $T_3$, respectively.
We remark that $d_1 \cup \iota (d_2) \cup d_3$ must be
smoothed out in the neighborhood of $\iota (S^1 \times \shuugou{0,1})$.
Then we have the followings.
\begin{itemize}
\item The set 
\begin{align*}
\bigcup_{J'} \varpi_{J'}(\skein{\Sigma^1, (J \cap \partial \Sigma^1) \cup \iota (J'  \times \shuugou{1})}
\times  \skein{S^1 \times I,J'\times \shuugou{0,1}}  \\
\times \skein{\Sigma^2,(J \cap \partial \Sigma^1) \cup \iota (J' \times \shuugou{0})})
\end{align*}
generates $\skein{\Sigma,J}$
as $\Q[A,\gyaku{A}]$-module.
\item The map $\varpi_{J'}$ preserves the filtrations, in other words,
\begin{align*}
 \varpi_{J'}(\skein{\Sigma^1, (J \cap \partial \Sigma^1) \cup \iota (J'  \times \shuugou{1})}
\times F^n  \skein{S^1 \times I,J'\times \shuugou{0,1}} \\
\times \skein{\Sigma^2,(J \cap \partial \Sigma^1) \cup \iota (J' \times \shuugou{0})}) 
\subset F^n \skein{\Sigma,J}.
\end{align*}
\item We have  $t_{c} \circ \varpi_{J'} = \varpi_{J'} \circ (\id ,t_{c_l},\id)$
and $\sigma(\iota(x)) \circ \varpi_{J'} = \varpi_{J'} \circ (\id,\sigma (x), \id)$
for $x \in \skein{S^1 \times I}$
\end{itemize}

We assume that $\Sigma \backslash \iota (S^1 \times (0,1))$ is a connected surface 
$\Sigma^1$.
For a finite set $J' \in S^1$, we consider the bilinear map
\begin{align*}
\varpi_{J'} : \skein{\Sigma^1, J \cup \iota (J'  \times \shuugou{0,1})}
\times  \skein{S^1 \times I,J'\times \shuugou{0,1}} 
\to \skein{\Sigma,J}
\end{align*}
defined by $\varpi_{J'} ([T_1],[T_2]) =[T_1T_2]$ for 
$T_1 \in \mathcal{T}(J \cup \iota (J'  \times \shuugou{0,1}))$ and
$T_2 \in \mathcal{T}(S^1 \times I,J'\times \shuugou{0,1})$.
Here we denote by $T_1T_2$ the tangle presented by
$d_1 \cup \iota (d_2) $, respectively, where 
$d_1$ and $d_2$ present $T_1$ and $T_2$, respectively.
We remark that $d_1 \cup \iota (d_2)$ must be
smoothed out in the neighborhood of $\iota (S^1 \times \shuugou{0,1})$.
Then we have the followings.
\begin{itemize}
\item The set 
\begin{align*}
\bigcup_{J'} \varpi_{J'}(\skein{\Sigma^1, J \cup \iota (J'  \times \shuugou{0,1})}
\times  \skein{S^1 \times I,J'\times \shuugou{0,1}})
\end{align*}
generates $\skein{\Sigma,J}$
as $\Q[A,\gyaku{A}]$-module.
\item The map $\varpi_{J'}$ preserves the filtrations, in other words,
\begin{align*}
 \varpi_{J'}(\skein{\Sigma^1, J \cup \iota (J'  \times \shuugou{0,1})}
\times F^n  \skein{S^1 \times I,J'\times \shuugou{0,1}} ) 
\subset F^n \skein{\Sigma,J}.
\end{align*}
\item We have $t_{c} \circ \varpi_{J'} = \varpi_{J'} \circ (\id ,t_{c_l})$
and $\sigma(\iota(x)) \circ \varpi_{J'} = \varpi_{J'} \circ (\id,\sigma (x))$
for $x \in \skein{S^1 \times I}$.
\end{itemize}
Hence, it suffices to show the following lemma. 

\begin{lemm}
\label{lemm_Dehn_twist}
Fix a positive integer $m$.
Choose points $p_1=\frac{1}{2m}$, $\cdots$, $p_i =
\frac{i}{2m}$, $\cdots$, $p_m =\frac{m}{2m}$ in $S^1$.
We denote by $r_i^0$ an element of
$\skein{S^1 \times I, \shuugou{(p_i,0),(p_i,1)}}$ represented by the tangle
presented by $\shuugou{p_i} \times I$.
Then we have the following.

\begin{itemize}
\item(1)We have $(t-1)^{2n+m}(r_1^0 \boxtimes r_2^0 
\boxtimes \cdots \boxtimes r_m^0) \in F^n\skein{S^1 \times I,\shuugou{p_1, \cdots,p_m}
\times \shuugou{0,1}}$. 
\item (2)We have
\begin{equation*}
\log (t) (r_1^0 \boxtimes r_2^0 
\boxtimes \cdots \boxtimes r_m^0) =
\sigma(\frac{-A+\gyaku{A}}{4\log(-A)}(\arccosh  (-\frac{l}{2}))^2)(r_1^0 \boxtimes r_2^0 
\boxtimes \cdots \boxtimes r_m^0
).
\end{equation*}
\end{itemize}
\end{lemm}

\begin{proof}[Proof of Lemma \ref{lemm_Dehn_twist}(1)]
We need  a $\Q [A,\gyaku{A}]$-bilinear map
$(\cdot)  (\cdot) :\skein{S^1 \times I, J \times \shuugou{0,1}}
\times \skein{S^1 \times I,J \times \shuugou{0,1}}
\to \skein{S^1 \times I,J \times \shuugou{0,1}}$ defined by
$[T_1][T_2] =[T_1T_2]$
for any finite subset $J \subset S^1$.
Here we denote by $T_1  T_2$ the tangle presented by
$\mu_1 (D_1) \cup \mu_2 (D_2)$ where 
we choose tangle diagrams $D_1$ and $D_2$
presenting $T_1$ and $T_2$, respectively,
and embeding maps $\mu_1$ and $\mu_2 :S^1 \times I 
\to S^1\times I$ defined by $\mu_1(\theta,t) =
(\theta,\frac{t+1}{2})$ and $\mu_2(\theta,\frac{t}{2})$.
We remark that $\mu_1 (D_1) \cup \mu_2 (D_2)$ must be
smoothed out in the neighborhood of $c_l$.
By definition we have $(F^k \skein{S^1 \times I,J \times \shuugou{0,1}})
 (F^l \skein{S^1 \times I,J \times \shuugou{0,1}}) \subset F^{k+l}
\skein{S^1 \times I,J \times \shuugou{0,1}}$ for $k$ and $l \in \Z_{\geq 0}$
and $x  y =y  x$ for $x$ and $y \in \skein{S^1 \times I,\shuugou{p} \times \shuugou{0,1}}$
for $p \in S^1$.
For $i =1, \cdots,m$, we denote by $x_i \defeq r_1^0 \boxtimes 
\cdots \boxtimes r_{i-1}^0 \boxtimes t(r_i^0) \boxtimes r_{i+1}^0
\boxtimes \cdots \boxtimes r_m^0$
and by ${x_i}^{-1} \defeq  r_1^0 \boxtimes 
\cdots \boxtimes r_{i-1}^0 \boxtimes t^{-1}(r_i^0) \boxtimes r_{i+1}^0
\boxtimes \cdots \boxtimes r_m^0$.
We simply denote $\id  \defeq r_1^0 \boxtimes r_2^0 
\boxtimes \cdots \boxtimes r_m^0$.
We remark that $x_i  {x_i}^{-1} =\id$.

We have 
\begin{align*}
&(t-1)^{2n+m} (\id) =(x_1x_2 \cdots x_m-\id)^{ 2n+m}=(\sum_{i=1}^m x_1  x_2  \cdots   x_{i-1}  (x_{i-1}-\id))^{ 2n+m}. \\ 
\end{align*}
Since $(t(r_i^0)-r_i^0)^{ 2} =-(l+2) t(r_i^0) +(A+1)t^2 (r_i^0) +(A^{-1}+1)r_i^0
\in F^1 \skein{S^1 \times I, p_i \times \shuugou{0,1}}$,
we have $(t-1)^{2n+m} (\id) \in F^n \skein{S^1 \times I,\shuugou{P_1, \cdots , P_m}
\times  \shuugou{0,1}}$.
This proves the part (1) of the lemma.
\end{proof}

To prove Lemma \ref{lemm_Dehn_twist} (2), we need the following lemma.

\begin{lemm}\label{lemm_Dehn_twist_toy}
We have $\sigma(\frac{-A+\gyaku{A}}{4\log(-A)}(\arccosh  (-\frac{l}{2}))^2)(r_i^0)
=\log (t)(r_i^0)$ for $i=1, \cdots,m$.
\end{lemm}

For $n=0,1,\cdots$, we define the Chebyshev polynomial $T_n (X) \in \Z [X]$
by setting $T_0 (X) =2$ and $T_{n+1} (X) =XT_n(X)-T_{n-1} (X)$.
We denote
by $(T+1)_n(X) \defeq \sum_{i=0}^n \frac{n!}{i!(n-i)!}T_i(X)$. 
It is obvious that $(T+1)_n(q+\gyaku{q}) =(q+1)^n+(\gyaku{q}+1)^n$.
Since
\begin{equation*}
(T+1)_n(x) =(\sqrt{x+2})^n((\frac{\sqrt{x+2}-\sqrt{x-2}}{2})^n+(\frac{\sqrt{x+2}+\sqrt{x-2}}{2})^n),
\end{equation*}
we have the following proposition.

\begin{prop}
We have
\begin{align*}
(T+1)_{2n} (x) &\in (x+2)^n \Z [x], \\
(T+1)_{2n+1} (x)  &\in (x+2)^n \Z[x].
\end{align*}
\end{prop}

We define a sequence $\shuugou{a_n}_{n \geq 2}$ by
$(\log (-x))^2 =\sum_{n=2}^\infty a_n (x+1)^n \in \Q [[x+1]]$.
We remark that $2(\arccosh   (-\frac{X}{2}))^2 =(\log (-T))^2 (X) \defeq
\sum_{n=2}^\infty a_n (T+1)_n (X) \in \Q [[X+2]]$.

\begin{proof}[Proof of Lemma \ref{lemm_Dehn_twist_toy}]
We have
\begin{align*}
&\sigma(\frac{-A+\gyaku{A}}{4\log(-A)}(\arccosh  (-\frac{l}{2}))^2)(r_i^0) \\
&=\sigma (\frac{-A+\gyaku{A}}{8 \log (-A)}(\log (-T))^2(l))(r_i^0) \\
&=\frac{1}{8 \log (-A)}((\log (-T))^2(l)(r_i^0)-(r_i^0)(\log (-T))^2(l)) \\
&=\frac{1}{8 \log (-A)}\sum_{k=1}^\infty (a_k (T+1)_k (l)r_i^0-a_k r_i^0(T+1)_k (l)) \\
&=\frac{1}{8 \log (-A)}\sum_{k=1}^\infty (a_k (-A r_i^1-r_i^0)^k
+a_k (-A^{-1} r_i^{-1}-r_i^0)^k-a_k (-A^{-1} r_i^1-r_i^0)^k-a_k (-A r_i^{-1}-r_i^0)^k) \\
&=\frac{1}{8 \log (-A)}(\log^2(-Ar_i^1) +\log^2(-\gyaku{A}\gyaku{r_i})-\log^2(-\gyaku{A}r_i^1)
-\log^2(-A \gyaku{r_i}) ) \\
&=\frac{1}{8 \log(-A)}(((2( \log (-A))^2r_i^0+2 \log(-A) \log r_i^1 + ( \log r_i^1)^2)) \\
&-2(( \log (-A))^2r_i^0-2 \log(-A) \log r_i^1 + ( \log r_i^1)^2)) \\
&=\log r_i^1=\log (t) (r_i^0).
\end{align*}
Here we denote by $r_i^1 \defeq t (r_i^0)$, by $\gyaku{r_i} \defeq \gyaku{t} (r_i^0)$,
$\log r_i^1 \defeq \sum_{i=1}^ \infty \frac{(-1)^{i-1}}{i} (r_i^1-r_i^0)$
and by $\log^2 ((-A)^{\epsilon_1} {r_i}^{\epsilon_2}) \defeq 
\sum_{n=2}^\infty a_n( (-A)^{\epsilon_1} {r_i}^{\epsilon_2}+r_i^0)^n$
for $\epsilon_1 \in \shuugou{-1,0,1}$
 and $\epsilon_2 \in \shuugou{-1,1}$.
This proves the lemma.
\end{proof}

\begin{proof}[Proof of Lemma \ref{lemm_Dehn_twist}(2)]
We have 
\begin{align*}
&\sigma(\frac{-A+\gyaku{A}}{4\log(-A)}(\arccosh  (-\frac{l}{2}))^2)(r_1^0 \boxtimes r_2^0 
\boxtimes \cdots \boxtimes r_m^0) \\
&=\sum_{i=1}^m r_{1}^0 \boxtimes \cdots
\boxtimes r_{i-1}^0 \boxtimes 
\sigma(\frac{-A+\gyaku{A}}{4\log(-A)}(\arccosh  (-\frac{l}{2}))^2)(r_i^0)
\boxtimes
r_{i+1}^0 \boxtimes \cdots \boxtimes r_{m}^0 \\
&=\sum_{i=1}^m r_{1}^0 \boxtimes \cdots
\boxtimes r_{i-1}^0 \boxtimes 
\log(t)(r_i^0)
\boxtimes
r_{i+1}^0 \boxtimes \cdots \boxtimes r_{m}^0 \\
& =\sum_{i=1}^m \log (x_i) \\
& =\log (x_1  x_2  \cdots  x_m) \\
&=\log(t) (\id).
\end{align*}
This proves the lemma.
\end{proof}

\begin{rem}
Let $\Sigma$ be a compact connected oriented surface with non-empty
connected boundary and let $\mathcal{I} (\Sigma) \subset  \mathcal{M} (\Sigma)$
be the Torelli group of $\Sigma$. Then we have

\begin{enumerate}
\item $\mathcal{I}(\Sigma) \subset  \check{\mathcal{M}} (\Sigma)$.
\item For any $\xi \in  \mathcal{I} (\Sigma) $,
there exists $x_{\xi} \in \widehat{\skein{\Sigma}}$ satisfying that
$x_{\xi}$ is a skein representative of $\xi \in \mathcal{I}(\Sigma) \subset
\check{\mathcal{M}}(\Sigma)
$ by $((\skein{\Sigma},\filtn{F^n\skein{\Sigma}}), \Theta (\Sigma))$
in the sense of Definition \ref{definition_quantization}.
\end{enumerate}

The proof will appear in \cite{TsujiCSAII}.
\end{rem}

\section{Filtrations}

\subsection{The filtrations depend only on the underlying $3$-manifold}

In this subsection, we prove the following theorem.
The proof of the theorem is analogous to that of
\cite{Ha2000} Proposition 6.10.

\begin{thm}
\label{thm_filtration_independent}
Let $\Sigma$ and $\Sigma'$ be two oriented compact connected surfaces,
 $J$ a finite subset of
$\partial \Sigma$ and $J'$ a finite subset of $\partial \Sigma'$
such that 
there exists a diffeomorphism $\xi : (\Sigma \times I ,J \times I)
 \to (\Sigma' \times I, J' \times I)$. 
Then we have $\xi (F^n \skein{\Sigma,J}) = F^n \skein{\Sigma' ,J'}$
for $n \geq 0$.
\end{thm}

To prove it, we need new filtrations of the Kauffman skein modules.

Let $\Q [A, \gyaku{A}]  \mathcal{T}
(\Sigma,J)$  be the free module with basis $\mathcal{T} (\Sigma,J)$ over
$\Q [A, \gyaku{A}]$ and  $\kukakko{ \cdot}$  the natural surjection
$\Q [A, \gyaku{A}]  \mathcal{T} (\Sigma,J) \to \skein{\Sigma,J}$.
For a tangle $ T \in  \mathcal{T} (\Sigma,J)$ and  closed components $L_1,L_2, 
\cdots, L_m$ of $T$, we denote by 
\begin{equation*} (T, \bigcup_{i=1}^m L_i)
\defeq  \sum_{j=0}^m \sum_{\shuugou{i_1,i_2, \cdots , i_j } \subset  \shuugou{1,2, \cdots, m}}
2^{m-j} T' \cup \bigcup_{h=1}^j L_{i_h} \in \Q[ A, \gyaku{A}] \mathcal{T} 
(\Sigma,J)
\end{equation*}
 where $T' \cup \bigcup_{i=1}^m L_i = T$.
We denote by $F^{\star 0} \skein{\Sigma,J}
\defeq \skein{\Sigma,J}$ and
by $F^{ \star n} \skein{\Sigma,J}$ the
$\Q [A, \gyaku{A}]$ submodule generated by 
$(A+1)F^{\star (n-1)}$ and the set of 
$\skein{\Sigma, J}$ consisting of the elements $\kukakko{(T, \bigcup_{i=1}^n K_i)}
\in \skein{\Sigma,J}$ for $T \in \mathcal{T} (\Sigma,J)$ and closed components
$K_1,K_2, \cdots, K_n$ of $T$ for $n \geq 1$.
Similarly, the filtration $\filtn{F^{\star n} \skein{\Sigma',J'}}$ is defined as 
$\filtn{F^{ \star n} \skein{\Sigma,J}}$.

\begin{lemm}
\label{lemm_filtration_prepare}
Let $E: D \sqcup  \coprod_{i=1}^n (S^1)_i \to \Sigma$ be an immersion
whose intersections  consist of 
transverse double points, and suppose that $E( \partial D) =J$.
Here we denote by $D= \coprod_{i'=1}^{\frac{\sharp J}{2}} I_{i'} \sqcup \coprod_{i''=1}^M
 (S^1)_{k+i''}$.
Fix an intersection $P$ of $E(D \sqcup \coprod_{i=1}^n (S^1)_i)$.
Let $d(1)$ and $d(2)$ be two tangle diagrams that equal $E(D \sqcup \coprod_{i=1}^m
(S^1)_i)$ with identical height-information except the neighborhood of $P$, where they 
looks  as in Figure \ref{fig_K2} and Figure \ref{fig_K1}, respectively.
We denote  by $T(1)$ and $T(2)$ the two tangles presented by $d(1)$ and $d(2)$,
 respectively.
For $i=1, \cdots, m$, let $K(1)_i$ be the component of $T(1)$ presented by 
$E((S^1)_i)$ and $K(2)_i$ the component of $T(2)$ presented by 
$E((S^1)_i)$.
Then we have $\kukakko{ (T(1), \bigcup_{i=1}^n K(1)_i)} - \kukakko{(T(2), \bigcup_{i=1}^n
K(2)_i)} \in (A+1) F^{\star (n-1)} \skein{\Sigma,J}$.

\end{lemm}

\begin{proof}
We denote by
\begin{align*}
T'(1) \defeq T(1) \backslash (\bigcup_{i \in \shuugou{1, 
\cdots, n}} K(1)_i ), \\
T'(2) \defeq T(2) \backslash (\bigcup_{i \in \shuugou{1, 
\cdots, n}} K(2)_i ). \\
\end{align*}
Let $T(0)$ and $T( \infty)$ be two elements of  $\mathcal{T} (\Sigma,J)$ presented by 
 tangle diagrams $d(0)$ and $d( \infty)$ which equal $d(1)$ except $D$, 
where they are shown in Figure \ref{fig_K0} and Figure \ref{fig_Kinfi}, respectively.

There are three cases.
\begin{enumerate}
\item  $P$ is a crossing of $K(1)_k$ and $K(1)_l$ for some 
$1 \leq k < l \leq m$.
\item $P$ is one of the crossings of $K(1)_k$ and the crossings of 
$K(1)_k$ and $T'(1)$ for some $1 \leq k \leq n$.
\item $P$ is a crossing of $T'(1)$
\end{enumerate}

(1)We assume $P$ is a crossing of $K(1)_k$ and $K(1)_l$ for some 
$1 \leq k < l \leq m$.
Let $K(0)_{kl}$ be the component of $T(0)$ satisfying 
$T(0) \backslash K(0)_{kl} =T(1) \backslash (K(1)_k \cup K(1)_l)$
and $K(\infty)_{kl}$ be the component of $T( \infty)$ satisfying
$T( \infty) \backslash K( \infty)_{kl} = T( 1) \backslash
(K(1)_k \cup K(1)_l)$.
For $i \neq k,l$, let $K(0)_i$ be the component of $T(0)$ presented by 
$E((S^1)_i)$ and $K(\infty)_i$ the component of $T(\infty)$ presented by 
$E((S^1)_i)$.
We denote by 
\begin{align*}
T'(0) \defeq T(0) \backslash (\bigcup_{i \in \shuugou{1, 
\cdots, n} \backslash \shuugou{k,l}} K(0)_i \cup K(0)_{kl}) ,\\
T'( \infty) \defeq T( \infty) \backslash (\bigcup_{i \in \shuugou{1, 
 \cdots, n} \backslash \shuugou{k,l}} K(\infty)_i \cup K(\infty)_{kl}).
\end{align*}
For a subset $\shuugou{i_1, \cdots, i_j} \subset \shuugou{1, \cdots, n}$
not including $k$ or $l$, we have
\begin{equation*}
\kukakko{T'(1) \cup \bigcup_{h \in \shuugou{i_1, \cdots, i_j}} K(1)_{i_h}}
-\kukakko{T'(2) \cup \bigcup_{h \in \shuugou{i_1, \cdots,i_j}}K(2)_{i_h}}=0.
\end{equation*}
Using Lemma \ref{lemm_kouten_sa}, for a subset 
$\shuugou{i_1, \cdots,i_j} \subset \shuugou{1, \cdots, n}$ including
$k$ and $l$, we have
\begin{align*}
\kukakko{T'(1) \cup \bigcup_{h \in \shuugou{i_1, \cdots, i_j}} K(1)_{i_h}}
-\kukakko{T'(2) \cup \bigcup_{h \in \shuugou{i_1, \cdots,i_j}}K(2)_{i_h}} \\
=(A-\gyaku{A}) (\kukakko{T'(0) \cup \bigcup_{h \in 
\shuugou{i_1, \cdots, i_j}\backslash \shuugou{k,l}} K(0)_{i_h}
\cup K(0)_{kl}} \\
-\kukakko{T'(0) \cup \bigcup_{h \in 
\shuugou{i_1, \cdots, i_j}\backslash \shuugou{k,l}} K(0)_{i_h}
\cup K(0)_{kl}}).
\end{align*}
Hence we have
\begin{align*}
\kukakko{(T(1), \bigcup_{i \in \shuugou{1, \cdots, n}} K(1)_i)}
-\kukakko{(T(2), \bigcup_{i \in \shuugou{1, \cdots, n}} K(2)_i)} \\
=(A-\gyaku{A})(
\kukakko{(T(0), (\bigcup_{i \in \shuugou{1, \cdots, n} \backslash 
\shuugou{k,l}} K(0)_i )\cup K(0)_{kl})} \\
-\kukakko{(T(\infty), (\bigcup_{i \in \shuugou{1, \cdots, n} \backslash
\shuugou{k,l}} K(\infty)_i) \cup K(\infty)_{kl})}.
\end{align*}

(2)We assume $P$ is one of the crossings of $K(1)_k$ and the crossings of 
$K(1)_k$ and $T'(1)$ for some $1 \leq k \leq n$. For $i \neq k$,
let $K(0)_i$ be the component of $T(0)$ presented by 
$E((S^1)_i)$ and $K(\infty)_i$ the component of $T(\infty)$ presented by 
$E((S^1)_i)$.
Using Lemma \ref{lemm_kouten_sa}, we have 
\begin{align*}
&\kukakko{(T(1), \bigcup_{i \in \shuugou{1, \cdots, n}} K(1)_i)}
-\kukakko{(T(2), \bigcup_{i \in \shuugou{1, \cdots, n}} K(2)_i)} \\
&=\kukakko{(T(1) \cup K(1)_k, \bigcup_{i \in \shuugou{1, \cdots, n}
\backslash \shuugou{k}} K(1)_i)}
-\kukakko{(T(2) \cup K(2)_k, \bigcup_{i \in \shuugou{1, \cdots, n}
\backslash \shuugou{k}} K(2)_i)} \\
&=(A-\gyaku{A})(
\kukakko{(T(0), \bigcup_{i \in \shuugou{1, \cdots, n} \backslash 
\shuugou{k}} K(0)_i)}
-\kukakko{(T(\infty), \bigcup_{i \in \shuugou{1, \cdots, n} \backslash
\shuugou{k}} K(\infty)_i )}.
\end{align*}

(3)We assume that $P$ is a crossing of $T'(1)$.
For $i=1, \cdots, n$, let $K(0)_i$ be the component of $T(0)$ presented by 
$E((S^1)_i)$ and $K(\infty)_i$ the component of $T(\infty)$ presented by 
$E((S^1)_i)$.
Using Lemma \ref{lemm_kouten_sa}, we have 
\begin{align*}
&\kukakko{(T(1), \bigcup_{i \in \shuugou{1, \cdots, n}} K(1)_i)}
-\kukakko{(T(2), \bigcup_{i \in \shuugou{1, \cdots, n}} K(2)_i)} \\
&=(A-\gyaku{A})(
\kukakko{(T(0), \bigcup_{i \in \shuugou{1, \cdots, n}} K(0)_i)}
-\kukakko{(T(\infty), \bigcup_{i \in \shuugou{1, \cdots, n} } K(\infty)_i )}.
\end{align*}

We conclude that
$\kukakko{ (T(1), \bigcup_{i=1}^n K(1)_i)} - \kukakko{(T(2), \bigcup_{i=1}^n
K(2)_i)} \in (A+1) F^{\star (n-1)} \skein{\Sigma,J}$.
This proves the lemma.

\end{proof}

\begin{lemm}
\label{lemm_new_filtration}
Let $\Sigma$ be a compact connected oriented surface,
and $J$ a finite subset of $\partial \Sigma$.
We have $F^n \skein{\Sigma,J} = F^{\star n} \skein{\Sigma,J}$
for any non-negative integer $n$.
Furthermore, We have
\begin{equation*}
\sum_{L' \subset L} (-1)^{\zettaiti{L'}}(-2)^{-\zettaiti{L'}}[L'] \in (\ker \epsilon)^n
\end{equation*}
 for a link $L$ in $\Sigma \times I$ 
having components more than $n$,
where the sum is over all sublinks $L' \subset L$
including the empty link and we denote by $\zettaiti{L}$
the number of components of $L$.
In other words, for a link $L$ in $\Sigma \times I$,
$(-1)^{\zettaiti{L'}}(-2)^{-\zettaiti{L'}}[L'] \mod (\ker \epsilon)^n$
is a finite type invariant of order $n$ in the sense of Le
\cite{Le1996} (3.2).
\end{lemm}

\begin{proof}
We prove the lemma by induction on $n$.
If $n=0$, we have $F^{\star 0}\skein{\Sigma,J} = F^0 \skein{\Sigma,J}
=\skein{\Sigma,J}$. We assume that
$n > 0 $ and $F^{n-1} \skein{\Sigma,J} = F^{\star (n-1)} \skein{\Sigma,J}$.
For any tangle $T \in \mathcal{T}(\Sigma,J)$ and knots $K_1, K_2,
 \cdots, K_n \in \mathcal{T} (\Sigma)$, we have
\begin{equation*}
(\kukakko{K_1} +2) (\kukakko{K_2}+2) \cdots (\kukakko{K_n} +2) \kukakko{T}
=\kukakko{(K_1 \boxtimes K_2 \boxtimes \cdots \boxtimes K_n \boxtimes
T, K_1 \boxtimes K_2 \boxtimes \cdots \boxtimes K_n)}.
\end{equation*}
Hence we have $F^{\star n}\skein{\Sigma,J} \supset F^n \skein{\Sigma,J}$.
Using Lemma \ref{lemm_filtration_prepare}, for any tangle $T$ and closed components
$K_1, K_2, \cdots ,K_n$ of $T$, we have
\begin{align*}
\kukakko{(K_1 \boxtimes K_2 \boxtimes \cdots \boxtimes K_n \boxtimes
T', K_1 \boxtimes K_2 \boxtimes \cdots \boxtimes K_n)}
-\kukakko{(T, \bigcup_{i=1}^n K_i)}  \\
\in (A-\gyaku{A}) F^{\star (n-1)} \skein{\Sigma,J}
=(A-\gyaku{A}) F^{n-1} \skein{\Sigma,J} \subset F^n \skein{\Sigma,J}.
\end{align*}
Here we denote by $T' \defeq T\backslash (\bigcup_{i=1}^n K_i)$.
Since $\kukakko{(K_1 \boxtimes K_2 \boxtimes \cdots \boxtimes K_n \boxtimes
T', K_1 \boxtimes K_2 \boxtimes \cdots \boxtimes K_n)}
=(\kukakko{K_1} +2) (\kukakko{K_2}+2) \cdots (\kukakko{K_n} +2) \kukakko{T'}
\in F^n \skein{\Sigma,J}$, we have  $\kukakko{(T, \bigcup_{i=1}^n K_i)} \in
F^n \skein{\Sigma,J}$. This proves the lemma.
\end{proof}

\begin{proof}[Proof of Theorem \ref{thm_filtration_independent}]
By the definition, we have $\xi (F^{\star n} \skein{\Sigma,J})
=F^{\star n} \skein{\Sigma', J'}$.
Using Lemma \ref{lemm_new_filtration}, we have
$\xi(F^{n} \skein{\Sigma,J}) =\xi (F^{\star n} \skein{\Sigma,J})
=F^{\star n} \skein{\Sigma', J'} =F^n \skein{\Sigma',J'}$.
This proves the theorem.
\end{proof}

In this paper, we define the Kauffman bracket $\mathcal{K} : \shuugou{
\textrm{unoriented framed links } S^3} \to \Q [A ,\gyaku{A}]$ by
$\kukakko{L} =\mathcal{K} (L)  \kukakko{\emptyset} \in \skein{I \times I}$
for $L \in \shuugou{
\mathrm{links \ \ in \ \ } S^3} =\mathcal{T} (I \times I)$.
For an unoriented framed link $L$ in $S^3$ and components
$K_1, K_2, \cdots, K_m$, we define 
$\mathcal{K}(L, \bigcup_{i=1}^m K_i) \defeq \sum_{j=0}^m
\sum_{\shuugou{i_1,i_2, \cdots, i_j} \subset \shuugou{1,2, \cdots, m}}
2^{m-j} \mathcal{K} (L' \cup \bigcup_{h=1}^j K_{i_h})$. Here 
we denote by $L' \defeq L \backslash \bigcup_{i=1}^m K_i$.

Using Lemma \ref{lemm_new_filtration}, we have the following corollary

\begin{cor}
\label{cor_Kauffman}
For an unoriented framed link $L$ in $S^3$ and some components
$K_1, K_2, \cdots, K_m$, we have 
$\mathcal{K}(L, \bigcup_{i=1}^m K_i) \in
(A+1)^m \Q [A , \gyaku{A}]$.
\end{cor}

\begin{proof}
Since $\ker \epsilon = (A+1)^m \Q [A,\gyaku{A}]\kukakko{\emptyset}$, 
we have $F^m \skein{I \times I} =(A+1)^m \Q [A,\gyaku{A}] \kukakko{\emptyset}$.
It is clear that $(A+1)^m \Q [A, \gyaku{A}] \kukakko{\emptyset}
=F^m \skein{I\times I} = F^{\star m} \skein{ I\times I}$.
We have $\kukakko{(L, \bigcup_{i=1}^m K_i)} =
\mathcal{K} (L, \bigcup_{i=1}^m K_i) \kukakko{\emptyset}$.
This proves the corollary.
\end{proof}

\subsection{Filtrations are Hausdorff}

In this subsection, we prove the following theorem.

\begin{thm}
\label{thm_Hausdorff}
Let $\Sigma$ be a compact connected oriented surface
with non-empty boundary and $J$
a finite subset of $\partial \Sigma$. We have
$\cap_{i=1}^\infty F^n \skein{\Sigma,J}=0$,
in other words, the natural homomorphism $\skein{\Sigma,J}
\to \widehat{\skein{\Sigma,J}}$ is injective.
\end{thm}

We denote by $\mathbb{V}$ be the subset of ${\Z_{\geq 0}} \times \Z_{\geq 0}
\times \Z_{\geq 0} $
consisting of the triples $(a,b,c) \in \Z_{\geq 0}^3$ which satisfies 
$a+b+c \in 2 \Z_{\geq 0}$ and $\zettaiti{b-c} \leq a \leq b+c$.
For $(a,b,c) \in \mathbb{V}$, we denote the following right figure
by the following left figure.

\begin{picture}(300,160)
\put(0,0){\includegraphics[width=120pt]{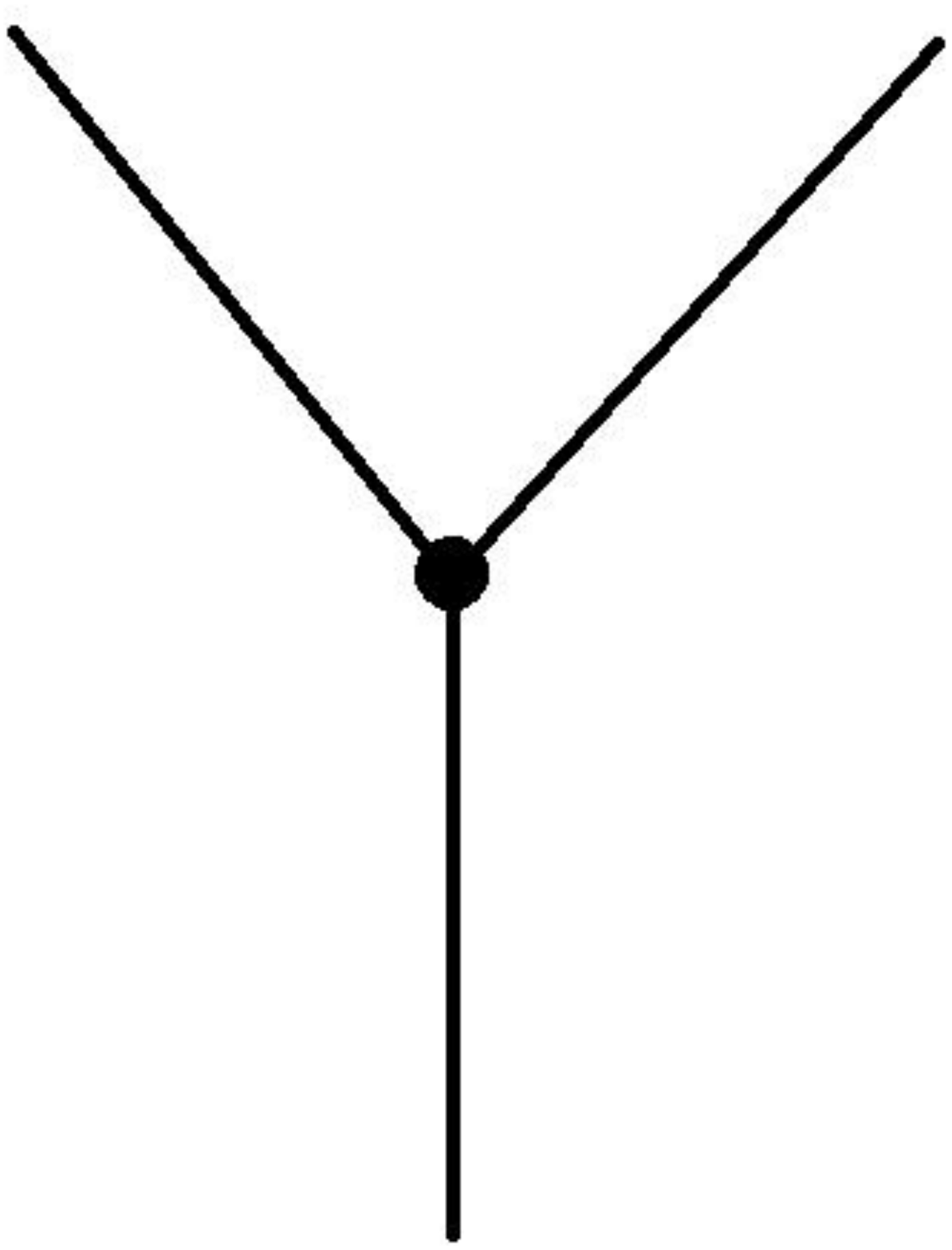}}
\put(120,0){\includegraphics[width=120pt]{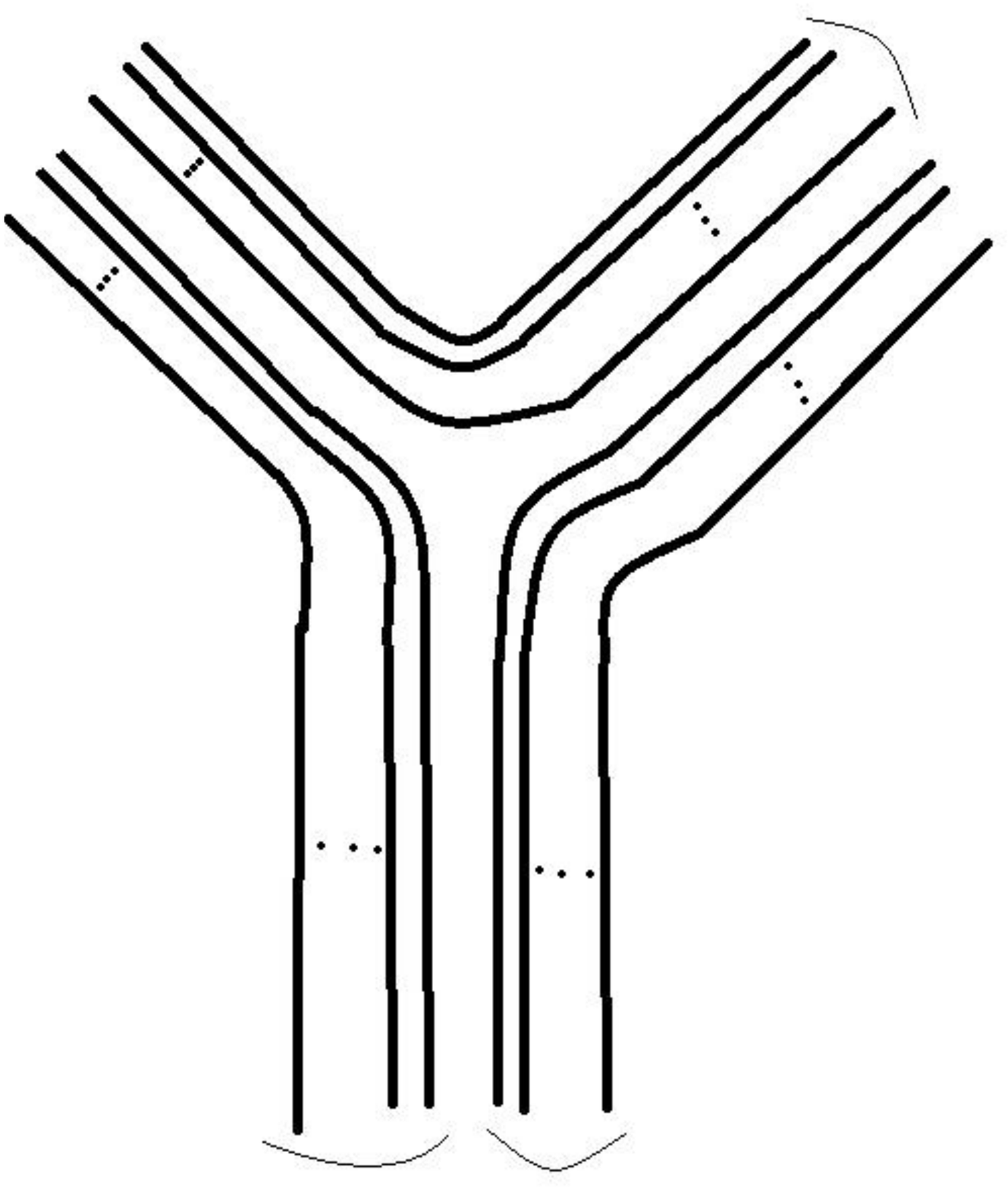}}
\put(60,40){$a$}
\put(30,120){$b$}
\put(75,120){$c$}
\put(124,10){$\frac{a+b-c}{2}$}
\put(190,10){$\frac{c+a-b}{2}$}
\put(220,130){$\frac{b+c-a}{2}$}
\end{picture}

Let $\Sigma_{0,g+1}$ be the surface $D^2 \backslash \coprod_{i=1}^g d_i$
where we denote by $D^2 = \shuugou{(x,y) \in \R^2 | x^2+y^2 \leq 1}$
and by $d_i$ the open disk $\shuugou{(x,y) \in \R^2 |x^2 + (y+1-\frac{i}{g+1})^2
< (\frac{1}{4g+4})^2}$ for $1 \leq i \leq g$.
We denote by $\mathbb{V} (g)$ the set of consisting of all
$(i_1,i_2, \cdots, i_{3g-3})$ which satifying
\begin{align*}
(i_{3j-3}, i_{3j-2}, i_{3j-1}), (i_{3j-1}. i_{3j}, i_{3j+1}) \in \mathbb{V} 
\end{align*}
for $j=1, \cdots ,g-1$. 
Here we define $i_0 \defeq  i_1$ and $i_{3g-2} \defeq
i_{3g-3}$. 
We denote by $\lambda(g,0) (i_1,i_2, \cdots ,i_{3g-3})$
the element of $\mathcal{T}(\Sigma_{0,g+1})$ presented by Figure \ref{figure_handlebody_basis}
for $(i_1,i_2, \cdots ,i_{3g-3}) \in \mathbb{V} (g)$.

\begin{figure}
\begin{picture}(240,320)
\put(0,0){\includegraphics[width=240pt]{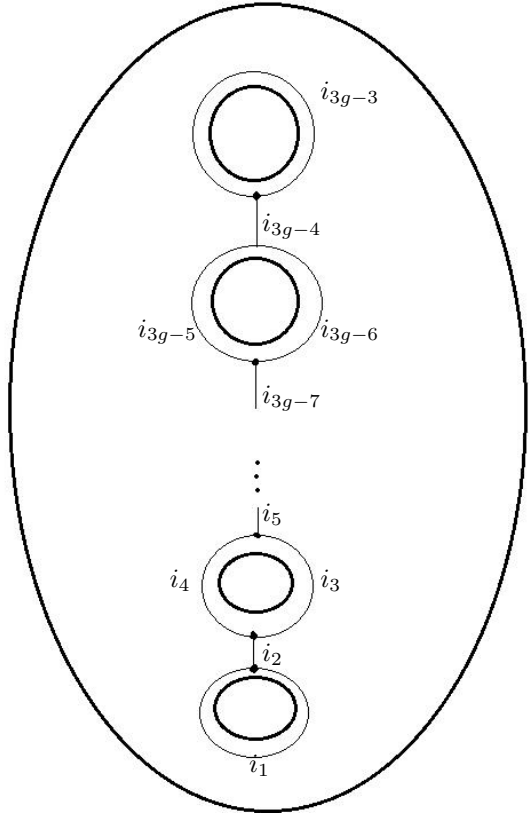}}
\put(120,15){$i_1$}
\put(125,57){$i_2$}
\put(147,85){$i_3$}
\put(90,85){$i_4$}
\put(125,110){$i_5$}
\put(125,155){$i_{3g-7}$}
\put(147,180){$i_{3g-6}$}
\put(78,180){$i_{3g-5}$}
\put(125,220){$i_{3g-4}$}
\put(147,270){$i_{3g-3}$}
\end{picture}
\caption{$\lambda(g,0) (i_1,i_2, \cdots ,i_{3g-3})$}
\label{figure_handlebody_basis}
\end{figure}

Fix an orientation preserving embedding $e_3: D^2 \times I \to S^3$
and a diffeomorphism $e_4 : \Sigma_{0,g+1} \times I \to 
\overline{S^3 \backslash e_3 (\Sigma_{0,g+1} \times I)}$
where we denote the closure of ${S^3 \backslash e_3 (\Sigma_{0,g+1} \times I)}$
by $\overline{S^3 \backslash e_3 (\Sigma_{0,g+1} \times I)}$.
Then we define a bilinear map 
$( \cdot, \cdot): \skein{\Sigma_{0,g+1}} \times \skein{\Sigma_{0,g+1}} \to \Q [A,\gyaku{A}]$
by $(\langle L_1 \rangle,\langle L_2 \rangle)=\mathcal{K} ( e_3(L_1) \cup e_4(L_2)
)$ for $L_1$ and $L_2 \in \mathcal{T} (\Sigma_{0,g+1})$.
The bilinear map induces 
$( \cdot, \cdot): \C \otimes \skein{\Sigma_{0,g+1}} \times 
\C \otimes \skein{\Sigma_{0,g+1}} \to \C [A,\gyaku{A}]$.
Here we denote by $\C [A,\gyaku{A}]$  the ring of Laurent polynomials over
$\C$. For a primitive $2r$-th root $\gamma$, the bilinear map induces
$( \cdot, \cdot): \mathcal{S}^\gamma (\Sigma_{0,g+1}) \times 
\mathcal{S}^\gamma (\Sigma_{0,g+1}) \to \C$ where we denote
by $\mathcal{S}^\gamma (\Sigma_{0,g+1}) \defeq 
\C \otimes \skein{\Sigma_{0,g+1}} /(A-\gamma)\C \otimes \skein{\Sigma_{0,g+1}}$.
The bilinear map induces the linear map
$\psi :\skein{\Sigma_{0,g+1}} \to \mathrm{Hom}_{\Q[A,\gyaku{A}]}(\skein{\Sigma_{0,g+1}},\Q
[A,\gyaku{A}])$ by $v \mapsto (u \mapsto (v,u))$. It induces the linear maps
\begin{align*}
\psi :\C \otimes \skein{\Sigma_{0,g+1}}  &\to
\mathrm{Hom}_{\C [A,\gyaku{A}]} (\C \otimes \skein{\Sigma_{0,g+1}},
\C [A,\gyaku{A}]), \\
\psi :\mathcal{S}^\gamma (\Sigma_{0,g+1}) &\to
\mathrm{Hom}_{\C} (\mathcal{S}^\gamma (\Sigma_{0,g+1}).
\end{align*}
We denote by $\cdot|_{A=\gamma}$
the quotient map $\C \otimes \skein{\Sigma_{0,g+1}} \to \mathcal{S}^\gamma
 (\Sigma_{0,g+1})$.

For a surface $\Sigma$ and any finite subset $J \subset \partial \Sigma$,
we recall that  we denote by $\mathcal{T}_0(\Sigma,J)$ the
set of isotopy classes of $1$-dimensional submanifolds of $\Sigma$
whose boundary is $J$ and that
$\skein{\Sigma,J} $ is freely generated by
$\mathcal{T}_0 (\Sigma,J)$ as a $\Q[A,\gyaku{A}]$-module.
 
\begin{thm}[Lickorish \cite{skeins_and_handlebodies}, P.347, Theorem]
\label{thm_Lickorish}
(1)The map $\lambda(g,0) :\mathbb{V} (g) \to \mathcal{T}_0 (\Sigma_{0,g+1})$ is
bijective.

(2)For a primitive $4r$-th root $\gamma$,
$\mathcal{S}^\gamma (\Sigma_{0,g+1}) / \ker \psi$ is a free $\C$-module with basis
\begin{align*}
\shuugou{\lambda(g,0) (i_1, \cdots , i_{3g-3}) |
(i_{3j-3}, i_{3j-2}, i_{3j-1}), (i_{3j-1}. i_{3j}, i_{3j+1}) \in \mathbb{V},  \\
 2r-4 \geq i_{3j-3}+i_{3j-2}+i_{3j-1}, 2r-4 \geq i_{3j-1}+i_{3j}+i_{3j+1}}.
\end{align*}
\end{thm}

We remark that Lickorish gave another basis in \cite{skeins_and_handlebodies}.
Using Theorem in \cite{skeins_and_handlebodies}, we have $\lambda(g,0)$ is injective.
It is proved in a similar way to the proof of Lemma \ref{lemm_lambda_surjective}
in this paper that $\lambda(g,0)$ is surjective.

\begin{lemm}
(1)The $ \C [A, \gyaku{A}]$ module homomorphism
$\psi :\C \otimes \skein{\Sigma_{0,g+1}}  \to
\mathrm{Hom}_{\C [A,\gyaku{A}]} (\C \otimes \skein{\Sigma_{0,g+1}},
\C [A,\gyaku{A}])$
 is injective.

(2)The $\Q [A,\gyaku{A}]$ module homomorphism
\begin{equation*}
\psi :\skein{\Sigma_{0,g+1}} \to
\mathrm{Hom}_{\Q [A,\gyaku{A}]} (\skein{\Sigma_{0,g+1}}, 
\Q [A, \gyaku{A}])
\end{equation*}
 is injective.
\end{lemm}

\begin{proof}
Let $x$ be an element of $\C \otimes \skein{\Sigma_{0,g+1}} \backslash \shuugou{0}$.
Using Theorem \ref{thm_Lickorish}, we have $\psi (x|_{A=\gamma}) \neq 0$ for some
 primitive $4r$-th root $\gamma$. In other words, we have
$(x|_{A=\gamma},y) \neq 0$ for some $y  \in \mathcal{S}^\gamma (\Sigma_{0,g+1})$. 
We regard $y$ as an element of $\C \otimes \skein{\Sigma_{0,g+1}}$ by 
$\mathcal{S}^\gamma (\Sigma_{0,g+1}) = \C \mathcal{T}_0 (\Sigma_{0,g+1}) \hookrightarrow 
\C [A,\gyaku{A}] \mathcal{T}_0 (\Sigma_{0,g+1}) =\C \otimes \skein{\Sigma_{0,g+1}}$.
Since $(x,y)|_{A=\gamma} =(x|_{A=\gamma},y) \neq 0$, we have
$(x,y) \neq 0$. This proves (1).

Let $x$ be an element of $ \skein{\Sigma_{0,g+1}} \backslash \shuugou{0}$.
We regard $x$ as an element of $\C \otimes \skein{\Sigma_{0,g+1}}$.
By (1), we have $(x,\sum_{j=1}^m (a_j+b_j \ii)c_j) \neq 0$
for some $a_j$ and 
$b_j \in \R$ and $c_j \in \mathcal{T}_0 (\Sigma_{0,g+1})$.
Let $k$ be an integer satisfying the coefficient of $A^k $
in $(x,\sum_{j=1}^m (a_j+b_j \ii)c_j)$ is not $0$.
We denote by $\omega (u_1, \cdots, u_m)$  the coefficient of $A^k$
in $(x, \sum_{j=1}^m u_j c_j)$ for $u_j \in \R$.
Then, $\omega :\R^m \to \R, (u_1, \cdots,u_m) \mapsto
\omega (u_1, \cdots, u_m)$ is a linear map.
Since $\omega$ is linear, $\omega$ is continuous.
By the definition, we have $\omega (a_1, \cdots, a_m) \neq 0$
or $\omega (b_1, \cdots, b_m) \neq 0$. By the density of $\Q$ in $ \R$, we have
$\omega (q_1, \cdots, q_m) \neq 0$ for some $q_1, \cdots, q_m \in \Q$.
Hence we obtain $(x, \sum_{j=1}^m q_j c_j) \neq 0$.
This proves (2).

\end{proof}

To prove Theorem \ref{thm_Hausdorff} in the case $J = \emptyset$,
we need the following lemma.

\begin{lemm}
\label{lemm_dual_filtration}
Let $\Sigma$ be a compact connected oriented surface
with non-empty boundary.
We have 
\begin{equation*}
\psi (F^k \skein{\Sigma}) =
\psi (F^{\star k} \skein{\Sigma})\subset
(A+1)^k \mathrm{Hom}_{\Q [A,\gyaku{A}]} (\skein{\Sigma}, 
\Q [A, \gyaku{A}]) 
\end{equation*}
for $k \in \Z_{ \geq 0}$.
\end{lemm}

\begin{proof}
By Theorem \ref{thm_filtration_independent},
it is sufficient to prove the lemma in the case $\Sigma = \Sigma_{0,g+1}$.
Let $L$ and $L'$ be links in 
$\Sigma_{0,g+1} \times I$ and $K_1, \cdots, K_k$ components of $L$.
By Corollary \ref{cor_Kauffman},
we have $\mathcal{K} (e_3(L) \cup e_4 (L'),\cup_{i=1}^k  e_3(K_i))
 \in (A+1)^k \Q [A,\gyaku{A}]$.
This proves the lemma.

\end{proof}

\begin{lemm}[A special case of Theorem \ref{thm_Hausdorff}]
\label{lemm_Hausdorff}
Let $\Sigma$ be a compact connected oriented surface
with non-empty boundary. We have
$\cap_{i=1}^\infty F^n \skein{\Sigma}=0$.
\end{lemm}

\begin{proof}
By Theorem \ref{thm_filtration_independent},
it is sufficient to prove the lemma in the case $\Sigma = \Sigma_{0,g+1}$.
By Lemma \ref{lemm_dual_filtration},
we have 
\begin{equation*}
\psi (\cap_{i=1}^\infty F^n \skein{\Sigma_{0,g+1}})
\subset \cap_{i=1}^\infty (A+1)^i \mathrm{Hom}_{\Q [A,\gyaku{A}]} (\skein{\Sigma}, 
\Q [A, \gyaku{A}]) =0.
\end{equation*}
Since $\psi$ is injective, we have
$\cap_{i=1}^\infty F^n \skein{\Sigma}=0$.
This proves the lemma.
\end{proof}

\bigskip

For $g \geq 1$ and $m \geq 1$, 
we denote by $\mathbb{V} (g,m)$ the set of consisting of
all $(i_1, i_2, \cdots, i_{2m-1}, j_1, j_2, \cdots, j_{3g-2})$ which 
satisfy $(i_{k-1},i_k,1) \in \mathbb{V}$ for $k=1, \cdots, 2m-1$,
$(j_{3k-2}, j_{3k-1}, j_{3k}), (j_{3k}, j_{3k+1}, j_{3k+2}) \in \mathbb{V}$
for $k=1, \cdots, g-1$ and $(i_{2n-1},j_1,j_2)$ where we define
by $i_0  \defeq 1$ and $ i_{3g-1} \defeq i_{3g-2}$.
Let $J$ be  a finite subset of $ \partial D^2 \subset \partial \Sigma_{0,g+1}$
satisfying $\sharp J =2m$.
We define a map $\lambda(g,m) :\mathbb{V} (g,m) \to
\mathcal{T}(\Sigma_{0,g+1},J)$ by 
$\lambda(g,m)(i_1, i_2, \cdots, i_{2m-1}, j_1, j_2, \cdots, j_{3g-2})$
looks as in Figure \ref{figure_tangle_basis} for any $(i_1, i_2, \cdots, i_{2m-1}, j_1, j_2, \cdots, j_{3g-2})
 \in \mathbb{V}(g,m)$.

\begin{figure}
\begin{flushleft}
\begin{picture}(200,240)
\put(0,-120){\includegraphics[width=360pt]{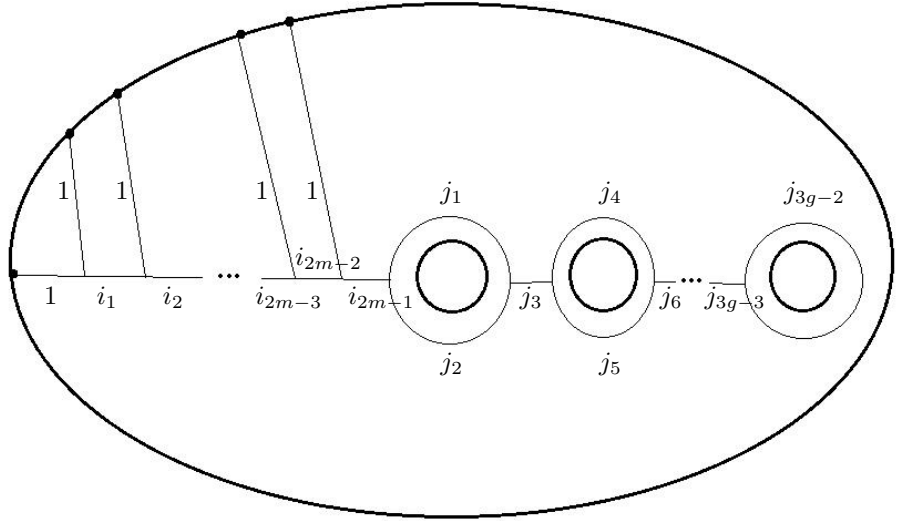}}
\put(20,95){$1$}
\put(40,95){$i_1$}
\put(65,95){$i_2$}
\put(100,95){$i_{2m-3}$}
\put(115,110){$i_{2m-2}$}
\put(135,95){$i_{2m-1}$}
\put(200,95){$j_3$}
\put(253,95){$j_6$}
\put(270,95){$j_{3g-3}$}
\put(170,135){$j_1$}
\put(230,135){$j_4$}
\put(300,135){$j_{3g-2}$}
\put(170,70){$j_2$}
\put(230,70){$j_5$}
\put(25,135){$1$}
\put(47,135){$1$}
\put(100,135){$1$}
\put(119,135){$1$}
\end{picture}
\end{flushleft}
\label{figure_tangle_basis}
\caption{$\lambda(g,m)(i_1, i_2, \cdots, i_{2m-1}, j_1, j_2, \cdots, j_{3g-2})$}
\end{figure}

\begin{lemm}
\label{lemm_lambda_surjective}
For $g \geq 1$ and $m \geq 1$, $\lambda (g,m):\mathbb{V} (g,m) \to
\mathcal{T}(\Sigma_{0,g+1},J)$ is surjective.
\end{lemm}

\begin{proof}
We use the following proposition.
For any $L \in \mathcal{T}_0 (\Sigma_{0,g+1},J)$, there exists $\tilde{L}$
representing $L$ and satisfying the above conditions in
Proposition \ref{prop_zensha_junbi}.
This proves the lemma.

\end{proof}

Let $\mathbb{I}_1, \cdots, \mathbb{I}_{2m-1},
\mathbb{J}_1, \cdots, \mathbb{J}_{3g-2}
$ be one-dimensional submanifolds  of $\Sigma_{0,g+1}$
 as in Figure \ref{figure_mathbb_L}. We denote by$\mathbb{L} \defeq
(\bigcup_{q}^{2m-1} \mathbb{I}_q)
\cup (\bigcup_{r}^{3g-2} \mathbb{J}_r)$.

\begin{figure}
\begin{flushleft}
\begin{picture}(200,240)
\put(0,-120){\includegraphics[width=360pt]{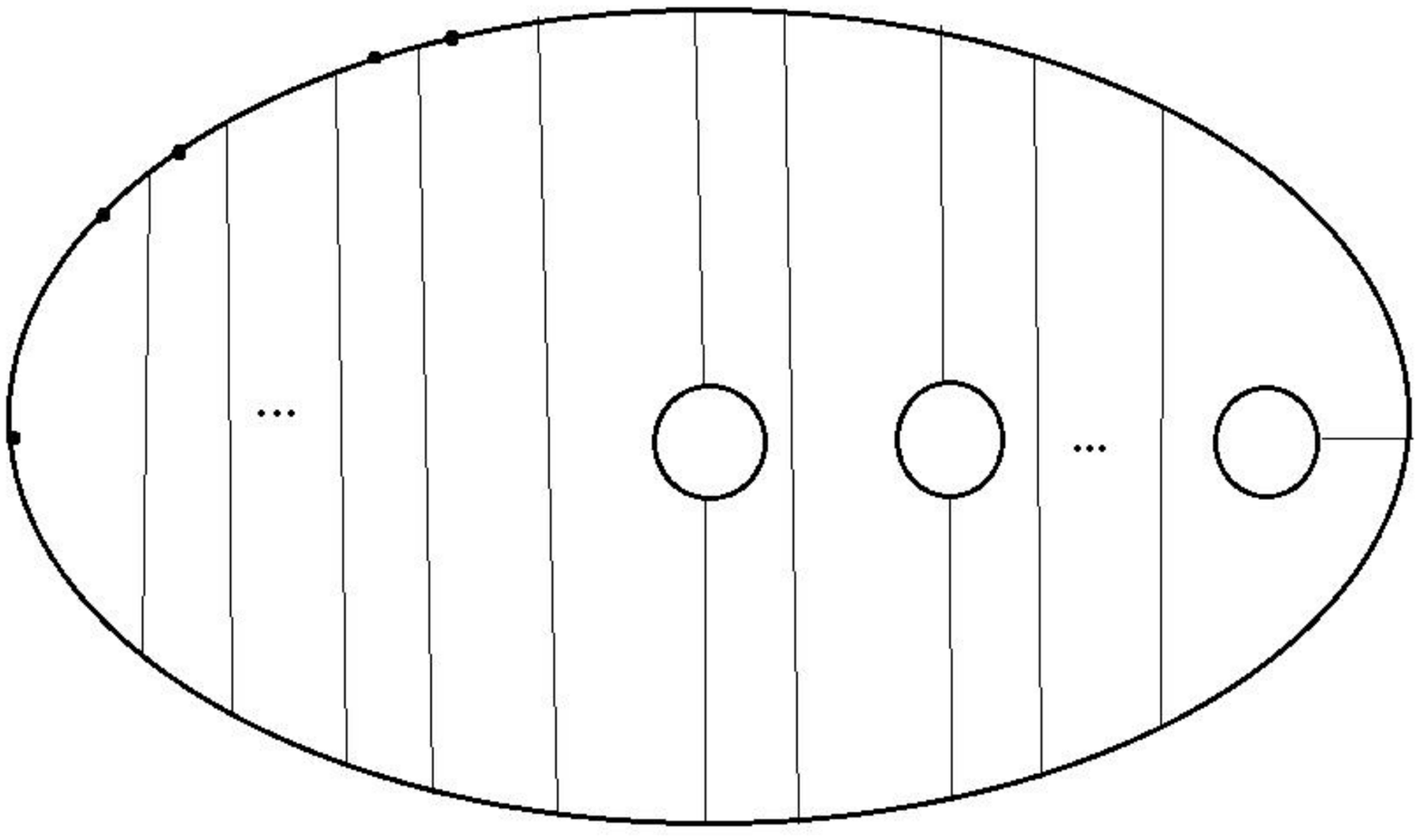}}
\put(40,95){$\mathbb{I}_1$}
\put(65,95){$\mathbb{I}_2$}
\put(85,95){$\mathbb{I}_{2m-3}$}
\put(108,110){$\mathbb{I}_{2m-2}$}
\put(135,95){$\mathbb{I}_{2m-1}$}
\put(195,95){$\mathbb{J}_3$}
\put(253,95){$\mathbb{J}_6$}
\put(270,95){$\mathbb{J}_{3g-3}$}
\put(170,135){$\mathbb{J}_1$}
\put(230,135){$\mathbb{J}_4$}
\put(317,117){$\mathbb{J}_{3g-2}$}
\put(170,70){$\mathbb{J}_2$}
\put(230,70){$\mathbb{J}_5$}
\end{picture}
\end{flushleft}
\caption{$\mathbb{L}$}
\label{figure_mathbb_L}
\end{figure}

We prove the following proposition by induction on $n$.

\begin{prop}[$n$]
\label{prop_zensha_junbi}
 Let $\tilde{L}$ be a one-dimensional submanifold of
$\Sigma_{0,g+1}$ satisfying the following conditions.
\begin{itemize}
\item There is no closed disk $d$ in 
$\Sigma_{0,g+1}$ such that $\partial d \subset \tilde{L}$.
\item We have $\partial \tilde{L} =J$.
\item The intersections $\tilde{L} \cap \mathbb{L}$
consist transverse double points.
\end{itemize}
We denote by $\alpha (\tilde{L}) $ the set consisting of all
$P \in \tilde{L} \cap \mathbb{L}$ satisfying the following condition
\begin{equation*}
e(\shuugou{(x,y) \in \partial D^2 |x \geq 0}) \subset \tilde{L},
e(\shuugou{(x,y) \in \partial D^2|x \leq 0}) \subset  \mathbb{L}, 
e(1,0) =P.
\end{equation*}
for some $e: D^2 \to \Sigma_{0,g+1}$,
Then we have $\sharp (\alpha (\tilde{L})) \leq n 
\Rightarrow L \in \lambda(g,m)(\mathbb{V}(g,m))$ where
we denote by $L$ the isotopy class of $\tilde{L}$.
\end{prop}

\begin{proof}
By definition, we have Proposition \ref{prop_zensha_junbi} ($0$).
We assume $n>0$ and Proposition \ref{prop_zensha_junbi} ($n-1$).
Let $\tilde{L}$ be a one-dimensional submanifold of $\Sigma_{0,g+1}$
satisfying the above conditions and $\sharp \alpha (\tilde{L}) =n$.
Since $\sharp \alpha (\tilde{L}) >0$, 
there exists an embedding $e: D^2 \to \Sigma_{0,g+1}$
such that $e(\shuugou{(x,y) \in \partial D^2 |x \geq 0}) \subset \tilde{L}$,
$e(\shuugou{(x,y) \in \partial D^2|x \leq 0}) \subset  \mathbb{L} \backslash \tilde{L}$.
Choose $\tilde{L}'$ a one-dimensional submanifold of 
$\Sigma_{0,g+1}$ which is $\tilde{L}$ except for the neighborhood of 
$e(D)$, where it looks as shown in the figure.

\begin{picture}(200,100)
\put(0,-50){\includegraphics[width=150pt]{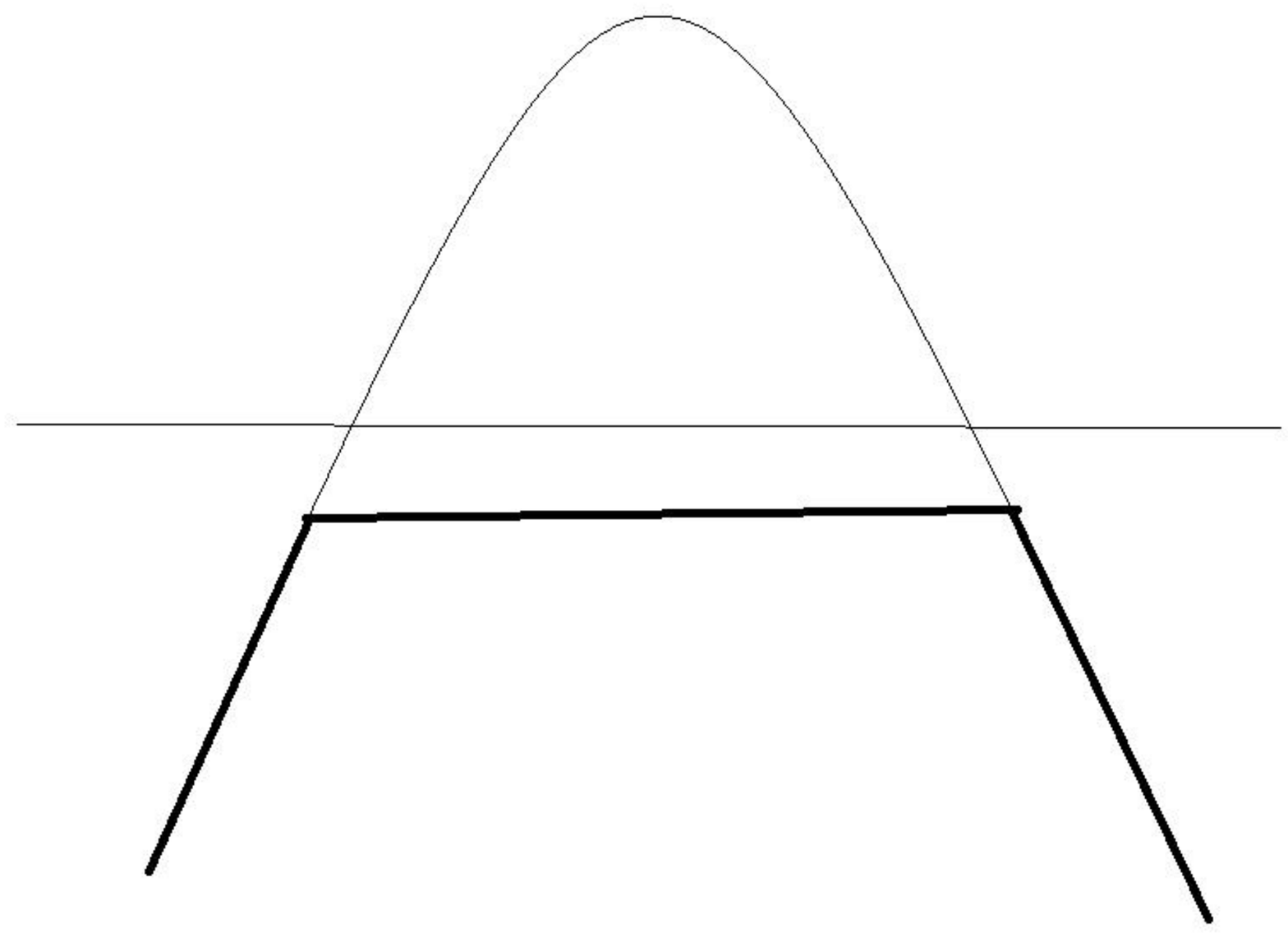}}
\put(73,63){$e(D)$}
\put(10,50){$\mathbb{L}$}
\put(57,70){$\tilde{L}$}
\put(43,20){$\tilde{L}'$}
\end{picture}

Since $\tilde{L} \simeq \tilde{L}'$ and $\sharp (\alpha(\tilde{L}')) <n$,
we have $L \in \lambda(g,m)(\mathbb{V}(g,m))$ where
we denote by $L$ the isotopy class of $\tilde{L}$.
This proves Proposition \ref{prop_zensha_junbi} ($n$) for any $n \geq 0$. 
\end{proof}

We define an injective map 
\begin{align*}
&\iota (g,m) :\mathbb{V} (g,m) \to \mathbb{V}(g+m) \\
&(i_1,i_2, \cdots,i_{2m-1}, j_1,j_2, \cdots,j_{3g-2}) \mapsto \\
&(1,i_1,1,i_2,i_3,1, \cdots,i_{2m-3},1,i_{2m-2},i_{2m-1},,j_1,j_2, \cdots,j_{3g-2}).
\end{align*}
for $m \geq 1$ and $g \geq 1$. 
Let $J$ be a finite subset of $\partial D^2 \subset \partial \Sigma_{0,g+1}$
satisfying $ \sharp J =2m$.
We define the $\Q[A,\gyaku{A}]$-module homomorphism 
$i(g,m):\skein{\Sigma_{0,g+1},J} \to
\skein{\Sigma_{0,g+m+1}}$ by
$\kukakko{\lambda(g,m)(v)} \to \kukakko{\lambda(g+m,0)(\iota(g,m)(v))}$.
Using Theorem \ref{thm_Lickorish}(1), we have the following proposition.

\begin{prop}
The $\Q[A,\gyaku{A}]$-module homomorphism 
$i(g,m):\skein{\Sigma_{0,g+1},J} \to
\skein{\Sigma_{0,g+m+1}}$
is well-defined and injective.
\end{prop}

\begin{cor}
The map $\lambda(g,m) :
\mathbb{V} (g,m) \to \mathcal{T}_0  (\Sigma_{0,g+1},J)$ is bijective.
\end{cor}

Let $\Sigma$ be a compact connected oriented surface with non-empty boundary,
$J$ a finite subset of $\partial \Sigma$ and $P_1$ and $P_2$ two points of
$J$. We choose two orientation preserving
embeddings $\delta_1,\delta_2 :I \to \partial \Sigma$
such that $\delta_1(I) \cap J =\delta_1 (\frac{1}{2}) =P_1$ and that
$\delta_2(I) \cap J =\delta_2 (\frac{1}{2}) =P_2$. 
We define a surface $\Sigma (P_1, P_2)$  by gluing
$\Sigma$ and $I \times I$ by $(0,1-t) =\delta_1 (t)$ and $(1,t) = \delta_2 (t)$.
We introduce $i'(P_1, P_2) :\mathcal{T}_0 (\Sigma,J) \to
\mathcal{T}_0 (\Sigma(P_1, P_2), J\backslash \shuugou{P_1,P_2})$ such
that $i'(P_1,P_2)(L)$ is the isotopy class of $\tilde{L} \cup \shuugou{
(\frac{1}{2},t) \in I \times I|t \in I}$ where $\tilde{L} $ represents $L$.
The map $i'(P_1,P_2)$ induces
a $\Q [A, \gyaku{A}]$-module homomorphism 
$i(P_1, P_2) :\mathcal{S} (\Sigma,J) \to
\mathcal{S} (\Sigma(P_1, P_2), J\backslash \shuugou{P_1,P_2})$.

\begin{lemm}
Let $\Sigma$ be a compact connected oriented surface with non-empty boundary,
$J$ a finite subset of $\partial \Sigma$ and $P_1$ and $P_2$ two points of
$J$. Then $i'(P_1, P_2) :\mathcal{T}_0 (\Sigma,J) \to
\mathcal{T}_0 (\Sigma(P_1, P_2), J\backslash \shuugou{P_1,P_2})$ is injective. 
In other words, $i(P_1, P_2) :\mathcal{S} (\Sigma,J) \to 
\mathcal{S} (\Sigma(P_1, P_2), $ $J \backslash \shuugou{P_1,P_2})$ is injective.
\end{lemm}

\begin{proof}
Let $J$ be $\shuugou{P_1,P_2, \cdots,P_{2m-1},P_{2m}}$.
We denote by $\eta \defeq i(P_{2m-1},P_{2m})
 \circ \cdots \circ i(P_3,P_4)\circ i(P_1,P_2) $
and by $\tilde{\Sigma} \defeq
\Sigma(P_1,P_2)(P_3,P_4) \cdots (P_{2m-1},P_{2m})$.
For some integer $g$ and some finite subset 
$J' \subset \partial D^2 \subset \partial \Sigma_{0,g+1}$,
we choose a diffeomorphism $\chi :(\Sigma \times I,
J \times I) \to (\Sigma_{0,g+1} \times I, J' \times I)$
and 
$\chi' :\tilde{\Sigma} \times I \to
\Sigma_{0,g+m+1} \times I$
satisfying $\chi_* \circ i(g,m) \circ (\chi'_*)^{-1} =
\eta$.
Here we denote by $\chi_*:\skein{\Sigma,J} \to
\skein{\Sigma_{0,g+1},J'}$ and $\chi'_*:
\skein{\tilde{\Sigma}} \to \skein{\Sigma_{0,g+m+1}}$ 
the $\Q[A, \gyaku{A}]$-module automorphisms induced by
$\chi$ and $\chi'$, respectively.
Since $i(g,m)$ is injective,
$\eta$ is also injective.
Hence $i(P_1,P_2)$ is injective.
This proves the lemma.

\end{proof}

\begin{proof}[Proof of Theorem \ref{thm_Hausdorff} in general cases]

We suppose $J \neq  \emptyset$.
Let $J$ be $\shuugou{P_1, \cdots,P_{2m}}$.
We denote by $\eta \defeq i(P_{2m-1},P_{2m})
 \circ \cdots \circ i(P_1,P_2) $
and by $\tilde{\Sigma} \defeq
\Sigma(P_1,P_2) \cdots (P_{2m-1},P_{2m})$.
By definition, we have $\eta (F^n \skein{\Sigma,J}) 
\subset F^n \skein{\tilde{\Sigma}}$ for any $n
\in \Z_{\geq 0}$.
Using Lemma \ref{lemm_Hausdorff}, we have
$\eta (\cap_{n=0}^ \infty F^n \skein{\Sigma,J})
\subset \cap_{n=0}^\infty F^n \skein{\tilde{\Sigma}} =0$.
Since $\eta$ is injective,
we have $\cap_{n=0}^ \infty F^n \skein{\Sigma,J}=0$.
This proves the theorem.

\end{proof}

\end{document}